\documentclass[11pt,reqno]{amsart}
\usepackage[T1]{fontenc}
\usepackage[utf8]{inputenc}
\usepackage[english]{babel}
\usepackage[margin=1in]{geometry}
\usepackage{amsmath,amssymb,amsthm,mathtools}
\usepackage{graphicx}
\usepackage{booktabs}
\usepackage{enumitem}
\usepackage{needspace}
\usepackage{hyperref}

\newtheorem{theorem}{Theorem}[section]
\newtheorem{lemma}[theorem]{Lemma}
\newtheorem{proposition}[theorem]{Proposition}
\newtheorem{corollary}[theorem]{Corollary}
\theoremstyle{definition}

\newtheorem{remark}[theorem]{Remark}
\newtheorem{example}[theorem]{Example}
\theoremstyle{plain}
\newtheorem{mainthm}[theorem]{Theorem}
\newtheorem{mainprop}[theorem]{Proposition}

\newcommand{\N}{\mathbb{Z}_{\ge 0}}
\newcommand{\Z}{\mathbb{Z}}
\newcommand{\Q}{\mathbb{Q}}
\newcommand{\C}{\mathbb{C}}
\newcommand{\R}{\mathbb{R}}
\newcommand{\dd}{\mathrm{d}}
\newcommand{\ee}{\mathrm{e}}

\hypersetup{colorlinks=true, linkcolor=blue, citecolor=blue, urlcolor=blue, pdftitle={From a symmetry between e and pi to the holonomic triangle}, pdfauthor={Beno\^it Cloitre}}
\numberwithin{equation}{section}

\title[From $\ee$ and $\pi$ to the holonomic triangle]{From a symmetry between $\ee$ and $\pi$ to the holonomic triangle}
\author{Benoit Cloitre}
\address{ORCID \href{https://orcid.org/0009-0001-6778-153X}{0009-0001-6778-153X}}
\date{}

\begin{document}

\begin{abstract}
In 2002, the author studied together the recurrences $u_{n+2}=u_{n+1}+u_n/n$ and $v_{n+2}=v_{n+1}/n+v_n$, in which the factor $1/n$ is placed in a symmetric way. The first one leads to $\ee$ and the second one to $\pi$, and each of them is classical when taken separately. When $n$ is replaced by $n+m$ in the denominators, with an integer $m\ge0$, the solutions grow like $f(m)\,n$ and $g(m)\sqrt n$, and the connection constants $f(m)$ and $g(m)$, suitably normalized, are explicit linear forms in $1$ and $\ee$ and in $1$ and $\pi$. We let $m$ become a complex variable $z$. Then $f$ and $g$ become meromorphic functions of $z$, the connection functions, which satisfy linear functional equations of order two in $z$, as the Gamma function satisfies $\Gamma(z+1)=z\Gamma(z)$. They add a third object to the classical pair formed by a P-recursive sequence and its D-finite generating function. We characterize them by their functional equations and their growth, as Wielandt's theorem does for the Gamma function, we show that they are differentially transcendental, and we explain the linear forms at the integers by the solutions excluded by the growth condition. In general, when the coefficients of a linear recurrence in $n$ depend on the parameter only through $n+z$, a translation argument gives a functional equation for the connection constant under the normalization and convergence assumptions of our translation theorem. Conversely, a solution of a linear functional equation with polynomial coefficients which has uniform polynomial growth in a right half-plane is the limit of an explicit P-recursive family of interpolation polynomials, and this applies to all the connection functions given by our existence theorem for coefficients of degree one in $n+z$.
\end{abstract}


\maketitle

\section{Introduction}

\subsection*{Notation}
We write $\N=\{0,1,2,\ldots\}$. The letters $n$ and $m$ always denote integers, $z$ denotes a complex variable, and $\tau=\Im z$. The logarithm $\log$ is the natural logarithm. A function $y$ holomorphic in a half-plane $\Re z>c$ has \emph{polynomial growth on vertical strips} if for all $c<a<b$ there exist $C>0$ and $N\ge0$ such that $|y(z)|\le C(1+|\tau|)^N$ in the strip $a\le\Re z\le b$. The Casoratian of two functions $y_1,y_2$ is $\mathcal{C}(y_1,y_2)(z)=y_1(z)y_2(z+1)-y_1(z+1)y_2(z)$.

\subsection{A symmetry between \texorpdfstring{$\ee$}{e} and \texorpdfstring{$\pi$}{pi}}\label{sec:intro-symmetry}

In 2002, the author studied together the two recurrences
\begin{equation}\label{eq:mirror}
u_{n+2}=u_{n+1}+\frac{u_{n}}{n},\qquad v_{n+2}=\frac{v_{n+1}}{n}+v_{n}\qquad(n\ge1),
\end{equation}
with $u_1=v_1=0$ and $u_2=v_2=1$, in which the factor $1/n$ is placed in a symmetric way. The first one leads to $\ee$ and the second one to $\pi$, namely
\[
\lim_{n\to\infty}\frac{n}{u_n}=\ee,\qquad \lim_{n\to\infty}\frac{2n}{v_n^2}=\pi .
\]
The interest of the pair lies in the simultaneous study of the two recurrences. Taken separately, and up to the choice of the initial values, each of them belongs to mathematical folklore. The first one is satisfied by $u_n=D_n/(n-1)!$, where $D_n$ is the number of derangements of $n$ objects. The second one is satisfied by $v_{2k}=v_{2k+1}=(2k-1)!!/(2k-2)!!$ for $k\ge1$, a quotient of double factorials related to the Wallis integrals. The sequences of the OEIS \cite{OEIS} which occur in the paper are listed in Section~\ref{sec:oeis}. This mirror symmetry between $\ee$ and $\pi$ was noticed at the time and recorded by Finch \cite[p.~19]{Finch2003}, by Gourdon and Sebah \cite{GourdonSebah}, and by Gourevitch \cite{Gourevitch}, and later by Loya, who derives the two limits in his book \cite[Section~7.7.3]{Loya2017}. We call the side of $u$ the \emph{world of $\ee$} and the side of $v$ the \emph{world of $\pi$}.

The author then introduced an integer parameter $m\ge0$ in the denominators,
\begin{equation}\label{eq:param}
u_{n+2}(m)=u_{n+1}(m)+\frac{u_{n}(m)}{n+m},\qquad v_{n+2}(m)=\frac{v_{n+1}(m)}{n+m}+v_{n}(m),
\end{equation}
with the same initial values, and found a more precise symmetry. We show in Section~\ref{sec:moments} that $u_n(m)\sim f(m)\,n$ and $v_n(m)\sim g(m)\sqrt n$ as $n\to\infty$, with $f(0)=1/\ee$ and $g(0)=\sqrt{2/\pi}$. We call $f(m)$ and $g(m)$ the \emph{connection constants} of the two families. For every $m\in\N$, the normalized constants $\ee f(m)$ and $\sqrt{\pi/2}\,g(m)$ are linear forms
\begin{equation}\label{eq:arith}
(-1)^m\,\ee f(m)=a_m-\ee\, b_m,\qquad (-1)^m\sqrt{\pi/2}\;g(m)=p_m-\pi q_m .
\end{equation}
In the world of $\ee$, the coefficients are factorials and derangement numbers,
\begin{equation}\label{eq:ab}
a_m=(m+1)!,\qquad b_m=D_{m+1}=(m+1)!\sum_{k=0}^{m+1}\frac{(-1)^k}{k!},
\end{equation}
and in the world of $\pi$, they are built from central binomial coefficients,
\begin{equation}\label{eq:pq}
p_{2k}=p_{2k+1}=\frac{4^k}{\binom{2k}{k}}\quad(k\ge0),\qquad q_0=0,\qquad q_{2k-1}=q_{2k}=\frac{k\binom{2k}{k}}{4^k}\quad(k\ge1).
\end{equation}
The coefficients of $\ee$ and of $\pi$ come from the sequences \eqref{eq:mirror} themselves, since $b_m=m!\,u_{m+1}$ and $q_{2k-1}=q_{2k}=v_{2k}/2$.
The first values are
\[
\renewcommand{\arraystretch}{1.5}
\begin{array}{c|ccccccc}
m & 0 & 1 & 2 & 3 & 4 & 5 & 6\\ \hline
(-1)^m\,\ee f(m) & 1 & 2-\ee & 6-2\ee & 24-9\ee & 120-44\ee & 720-265\ee & 5040-1854\ee\\
(-1)^m\sqrt{\pi/2}\;g(m) & 1 & 1-\frac{\pi}{2} & 2-\frac{\pi}{2} & 2-\frac{3\pi}{4} & \frac83-\frac{3\pi}{4} & \frac83-\frac{15\pi}{16} & \frac{16}{5}-\frac{15\pi}{16}
\end{array}
\]
Since $f(m)$ and $g(m)$ tend to $0$ as $m\to\infty$, each line gives rational approximations. The first line gives $(m+1)!/D_{m+1}\to\ee$, which comes from the series $\sum_{k\ge0}(-1)^k/k!=1/\ee$, whose partial sums are $D_n/n!$. The second line gives $p_m/q_m\to\pi$. For $m=2k$ this ratio is $16^k/\bigl(k\binom{2k}{k}^2\bigr)$, and the coefficients satisfy $p_mq_m=m/2$. The first approximation converges factorially and the second one like $1/m$, for a reason given in Section~\ref{sec:paths}.

This symmetry raises the question of its meaning. As functions of a complex variable, the two connection constants are of a different nature, a function of incomplete Gamma type for $\ee$ and a difference of quotients of Gamma functions for $\pi$, so the symmetry does not extend as such. It leads, however, to a common treatment of the two recurrences once $m$ is replaced by a complex variable $z$. The two connection constants then become meromorphic functions of $z$ which satisfy functional equations of order two, as the Gamma function satisfies $\Gamma(z+1)=z\Gamma(z)$, and the linear forms \eqref{eq:arith} become the values at the integers of a decomposition of these functions.

In a note of 2004 \cite{Cloitre2004}, the author studied two ways to introduce parameters into \eqref{eq:mirror}. When the parameters enter as coefficients, the asymptotics involve the Gamma function at the sum of the parameters, multiplied by exponentials in the parameters (Section~\ref{sec:branch}). When a parameter enters as a shift of $n$, as in \eqref{eq:param}, the note gives \eqref{eq:arith} and the first-order equation satisfied by the connection constant in the world of $\ee$ \cite[Part~A, items~3 and~5]{Cloitre2004}, but it does not study this second way further. The present paper does so.

\subsection{A third object for a classical pair}

A sequence is P-recursive when it satisfies a linear recurrence with polynomial coefficients, and a power series is D-finite when it satisfies a linear differential equation with polynomial coefficients. A sequence is P-recursive if and only if its generating function is D-finite \cite{Stanley1980}, so that the recurrence in $n$ and the differential equation in $x$ describe the same object, and the passage from one to the other is standard \cite[Chapter~7]{KauersPaule2011}. When the recurrence depends on a parameter $z$, an asymptotic behavior $u_n(z)\sim\Lambda(n,z)\,h(z)$, with an explicit normalizing factor $\Lambda(n,z)$, gives a third object, the connection constant $h$ seen as a function of $z$. One may then ask whether $h$ satisfies a linear functional equation $\sum_{k=0}^{r}R_k(z)\,h(z+k)=0$ with rational coefficients $R_k$. Euler's Gamma function is the simplest case, with the sequence $w_n(z)=n!/(z)_n$, where $(z)_n=z(z+1)\cdots(z+n-1)$ is the Pochhammer symbol, which satisfies $w_n(z)\sim\Gamma(z)\,n^{1-z}$. Its connection constant is $\Gamma(z)$, and its functional equation is $\Gamma(z+1)=z\Gamma(z)$.

Figure~\ref{fig:triangle} shows the three objects and the passages between them. We call it the \emph{holonomic triangle}, and we call its top the \emph{new branch}. The dashed arrows which lead to the new branch are given by Theorems~\ref{thm:T} and~\ref{thm:F} and Proposition~\ref{prop:D}, and Theorems~\ref{thm:G} and~\ref{thm:N} give constructions in the opposite direction, each under its own hypotheses. Table~\ref{tab:examples} gives the three objects for the Gamma function and for the recurrences \eqref{eq:param} with $m$ replaced by $z$. In the two worlds, the functional equations are of order two.

\begin{figure}[t]
\centering
\includegraphics[width=0.96\textwidth]{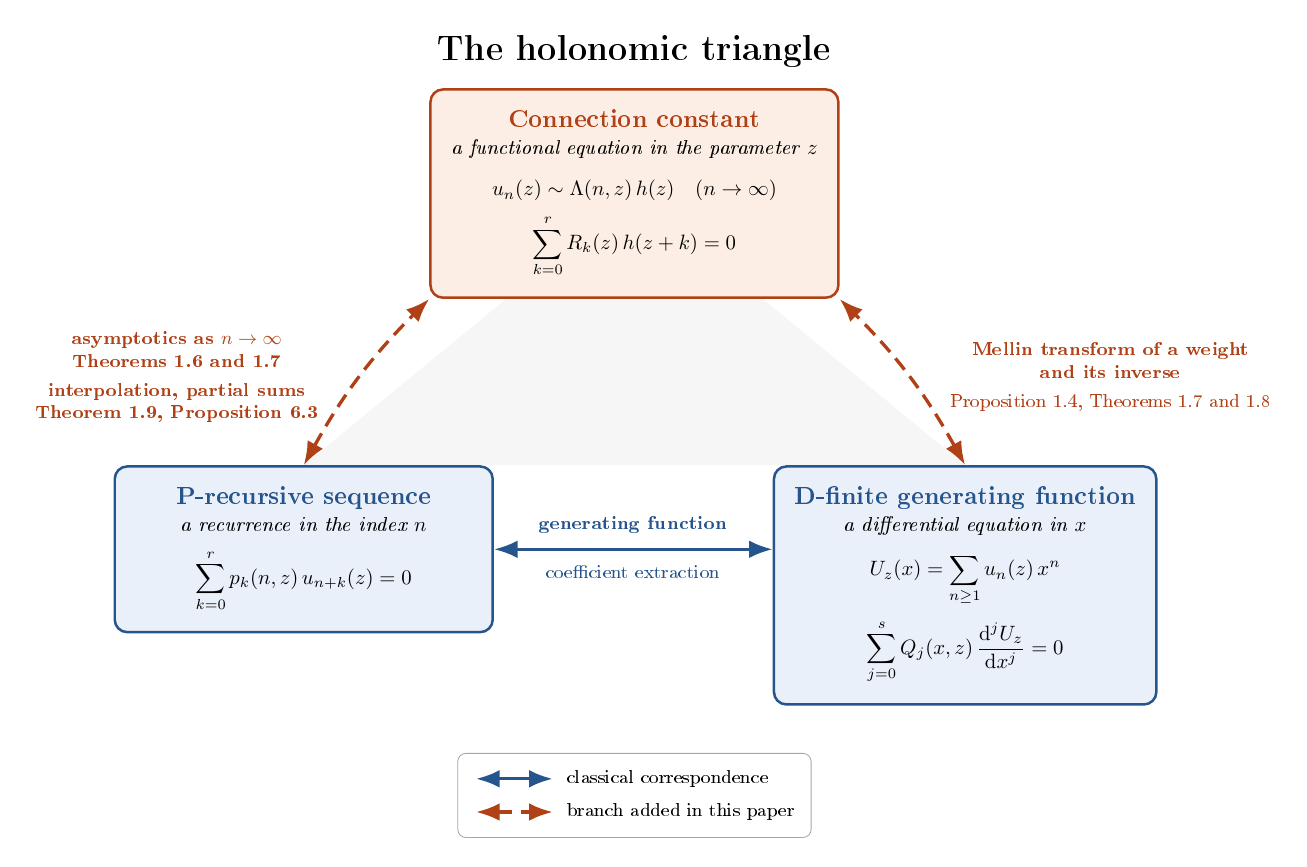}
\caption{The holonomic triangle. The dashed arrows are the contributions of this paper, Theorems~\ref{thm:T} and~\ref{thm:F} upwards, and Theorems~\ref{thm:G} and~\ref{thm:N} and Proposition~\ref{prop:partial} downwards, under different hypotheses.}\label{fig:triangle}
\end{figure}

\begin{table}[ht]
\centering\footnotesize
\renewcommand{\arraystretch}{2.2}
\setlength{\tabcolsep}{5pt}
\resizebox{\textwidth}{!}{%
\begin{tabular}{@{}llll@{}}
\toprule
 & Gamma function & world of $\ee$ & world of $\pi$\\
\midrule
recurrence in $n$ & $w_{n+1}=\dfrac{n+1}{n+z}\,w_n$ & $u_{n+2}=u_{n+1}+\dfrac{u_n}{n+z}$ & $v_{n+2}=\dfrac{v_{n+1}}{n+z}+v_n$\\
equation in $x$ & $x(1-x)W'+(z-1-x)W=x$ & $\begin{aligned}x(1-x)U'&=\bigl(x^2+(z-1)x+2-z\bigr)U\\[-2pt]&\quad+zx^2\end{aligned}$ & $\begin{aligned}x(1-x^2)V'&=\bigl(zx^2+x+2-z\bigr)V\\[-2pt]&\quad+zx^2\end{aligned}$\\
asymptotics & $w_n\sim\Gamma(z)\,n^{1-z}$ & $u_n\sim f(z)\,n$ & $v_n\sim g(z)\sqrt n$\\
constant & $\Gamma(z)=\displaystyle\int_0^\infty t^{z-1}\ee^{-t}\,\dd t$ & $\ee f(z)=z\displaystyle\int_0^1t^{z-1}(1-t)\,\ee^{t}\,\dd t$ & $\sqrt{\dfrac{\pi}{2}}\,g(z)=z\displaystyle\int_0^1t^{z-1}\sqrt{\dfrac{1-t}{1+t}}\,\dd t$\\
equation in $z$ & $\Gamma(z+1)=z\,\Gamma(z)$ & $f(z+2)=(z+2)\bigl(f(z)-f(z+1)\bigr)$ & $g(z+2)=g(z)-\dfrac{g(z+1)}{z+1}$\\
\bottomrule
\end{tabular}}
\caption{The three examples of Figure~\ref{fig:triangle}, with $w_1=1/z$, $u_1=v_1=0$, and $u_2=v_2=1$. The generating functions are $W$, $U=U_z$, and $V=V_z$, and the integrals hold for $\Re z>0$.}\label{tab:examples}
\end{table}

\subsection{Connection functions as moments}

For a weight $\omega$ on $[0,1]$ put $M_\omega(s)=\int_0^1 t^{s}\omega(t)\,\dd t$. The two worlds use the weights $\omega_{\ee}(t)=\ee^{t}$ and $\omega_{\pi}(t)=(1-t^2)^{-1/2}$, combined by the same operation
\begin{equation}\label{eq:operation}
h_\omega(z)=z\bigl(M_\omega(z-1)-M_\omega(z)\bigr)=z\int_0^1 t^{z-1}(1-t)\,\omega(t)\,\dd t\qquad(\Re z>0).
\end{equation}
The next theorem shows that $f(m)$ and $g(m)$ are values of these moments. It extends $f$ and $g$ to meromorphic functions of $z$, which we call the \emph{connection functions}, and it gives their closed forms, their functional equations, and their growth.

\begin{mainthm}\label{thm:A}
For every $m\in\N$,
\[
\ee f(m)=h_{\omega_\ee}(m),\qquad \sqrt{\pi/2}\;g(m)=h_{\omega_\pi}(m),
\]
where at $m=0$ the right-hand sides are understood as limits as $z\to0^+$, both equal to $1$. The functions defined for $\Re z>0$ by $\ee f=h_{\omega_\ee}$ and $\sqrt{\pi/2}\,g=h_{\omega_\pi}$ extend to meromorphic functions on $\C$ given by
\begin{equation}\label{eq:closed}
f(z)=\frac1\ee\sum_{k\ge0}\frac{1}{k!\,(z+k+2)},\qquad g(z)=\sqrt2\,\bigl(A(z)-A(z-1)\bigr),\qquad A(z)=\frac{\Gamma(z/2+1)}{\Gamma((z+1)/2)}.
\end{equation}
The function $f$ has a simple pole with residue $1/(\ee\,k!)$ at $z=-k-2$ for every $k\in\N$, and $g$ has simple poles at the negative integers. They satisfy, identically on $\C$,
\begin{gather}
f(z+1)+(z+2)f(z)=1,\label{eq:fe-f1}\\
f(z+2)=(z+2)\bigl(f(z)-f(z+1)\bigr),\label{eq:fe-f}\\
(z+1)\,g(z+2)=(z+1)\,g(z)-g(z+1).\label{eq:fe-g}
\end{gather}
The function $f$ has polynomial growth on vertical strips in $\Re z>-2$, and $g$ in $\Re z>-1$, and both are bounded on every vertical strip for $|\tau|\ge1$. As $|z|\to\infty$ in a sector $|\arg z|\le\pi-\varepsilon$, we have $f(z)\sim1/z$ and $g(z)\sim1/(2\sqrt z)$.
\end{mainthm}

The series for $f$ shows that $\ee f(z-2)=\int_0^1t^{z-1}\ee^{t}\,\dd t$. The formula for $g$ comes from $M_{\omega_\pi}(s)=\tfrac{\sqrt\pi}{2}\,\Gamma(\tfrac{s+1}{2})/\Gamma(\tfrac s2+1)$. Table~\ref{tab:mirror} compares the two worlds, with the objects introduced in the next two subsections.

\subsection{A characterization in the spirit of Wielandt}

Wielandt's theorem characterizes $\Gamma$ as the only solution of $F(z+1)=zF(z)$ that is holomorphic in $\Re z>0$, bounded in the strip $1\le\Re z\le2$, and equal to $1$ at $z=1$ \cite{Remmert1996}. The boundedness can be weakened to a moderate growth condition \cite{Fuglede2008}. The next theorem gives the analogue for the two connection functions.

\needspace{9\baselineskip}
\begin{mainthm}\label{thm:B}
Let $y$ be a meromorphic function on $\C$ which is holomorphic and of polynomial growth on vertical strips in some half-plane $\Re z>c$.
\begin{enumerate}[label=(\roman*)]
\item If $y(z+2)=(z+2)\bigl(y(z)-y(z+1)\bigr)$, then $y=\lambda f$ for some $\lambda\in\C$.
\item If $(z+1)\,y(z+2)=(z+1)\,y(z)-y(z+1)$, then $y=\lambda g$ for some $\lambda\in\C$.
\end{enumerate}
\end{mainthm}

Neither hypothesis can be removed. The functions
\[
f(z)+\frac{\Gamma(z+2)}{\sin\pi z},\qquad g(z)+\frac{A(z)+A(z-1)}{\sin\pi z}
\]
satisfy \eqref{eq:fe-f} and \eqref{eq:fe-g}, respectively, and have polynomial growth on vertical strips for $|\tau|\ge1$, but they have poles at the integers. The functions
\[
f(z)+\frac{1}{\Gamma(-1-z)},\qquad g(z)+\bigl(A(z)+A(z-1)\bigr)\sin\pi z
\]
satisfy the same equations and are holomorphic for $\Re z>-2$ and $\Re z>-1$, respectively. They coincide with $f$ on $\{-1,0,1,\ldots\}$ and with $g$ on $\N$, and they have exponential type $\pi/2$ and $\pi$, respectively, on vertical strips. Hence the values at the integers and the functional equations do not determine $f$ and $g$. Polynomial growth can be relaxed to a bound $C(1+|\tau|)^N\ee^{\alpha|\tau|}$ on vertical strips, with $\alpha<\pi/2$ for $f$ and $\alpha<\pi$ for $g$, and the last two examples show that these thresholds are sharp (Proposition~\ref{prop:types}). In both worlds, the growth condition excludes a solution which contains the factor $\ee^{i\pi z}$ and grows exponentially in one vertical direction. We call it the \emph{discarded solution} (Table~\ref{tab:mirror}). As Euler's Gamma function, which satisfies no algebraic differential equation by a theorem of H\"older, the functions $f$ and $g$ are differentially transcendental over $\C(z)$ (Proposition~\ref{prop:hypertr}), by criteria of Hardouin and Singer.

\begin{table}[b]
\centering\small
\renewcommand{\arraystretch}{1.3}
\begin{tabular}{@{}lll@{}}
\toprule
 & world of $\ee$ & world of $\pi$\\
\midrule
discarded solution & $y_2(z)=\ee^{i\pi z}\,\Gamma(z+2)$ & $y_2(z)=\ee^{i\pi z}\bigl(A(z)+A(z-1)\bigr)$\\
Casoratian & $\mathcal{C}(y_1,y_2)(z)=-\ee^{i\pi z}\,\Gamma(z+2)$ & $\mathcal{C}(y_1,y_2)(z)=-\sqrt2\,\ee^{i\pi z}$\\
Newton polygon & slopes $0$ and $1$ & slope $0$, roots $\pm1$\\
values at $m\in\N$ & $\ee f(m)\in\Z+\ee\Z$ & $\sqrt{\pi/2}\,g(m)\in\Q+\pi\Q$\\
\bottomrule
\end{tabular}
\caption{The two worlds. Here $y_1$ is the connection function, $f$ or $g$, and $y_2$ is the discarded solution.}\label{tab:mirror}
\end{table}

\subsection{The linear forms at the integers}

The next theorem shows that the decompositions \eqref{eq:arith} come from the discarded solutions.

\begin{mainthm}\label{thm:C}
\leavevmode
\begin{enumerate}[label=(\roman*)]
\item Let $H(z)=\ee^{i\pi z}\,\Gamma(z+2)$ and $E(z)=H(z)-\ee f(z)$. Then $E$ is entire, $H$ and $E$ satisfy \eqref{eq:fe-f}, and $E(z+1)+(z+2)E(z)=-\ee$. For every $m\in\N$,
\[
H(m)=(-1)^m a_m,\qquad E(m)=(-1)^m\,\ee\, b_m ,
\]
so that the first identity in \eqref{eq:arith} is the restriction to $\N$ of $\ee f=H-E$.
\item Let $\ell=\sqrt{\pi/2}\,g$ and $S_{\pm}(z)=\frac12\bigl(\ell(z)\pm\sqrt\pi\,\ee^{i\pi z}\bigl(A(z)+A(z-1)\bigr)\bigr)$. Then $S_+$ and $S_-$ satisfy \eqref{eq:fe-g}, $\ell=S_++S_-$, and for every $m\in\N$
\[
(-1)^m S_+(m)=p_m,\qquad (-1)^m S_-(m)=-\pi q_m .
\]
Equivalently $p_m+\pi q_m=\sqrt\pi\,\bigl(A(m)+A(m-1)\bigr)$, so that the conjugation $\pi\mapsto-\pi$ exchanges the normalized values at the integers of the connection function and of the discarded solution.
\end{enumerate}
None of the functions $H$, $E$, $S_+$, $S_-$ has polynomial growth on vertical strips.
\end{mainthm}

In each world the normalized connection function, $\ee f$ or $\sqrt{\pi/2}\,g$, is therefore the sum of two solutions of exponential growth on vertical strips. One of them has rational values at the integers, the other one has values in $\ee\Q$ or in $\pi\Q$, and the growth condition of Theorem~\ref{thm:B} isolates their sum. The closed formulas for $p_m$ and $q_m$ and their alternation follow directly, because $\sqrt\pi\,A(m)$ is rational for even $m$ and lies in $\pi\Q$ for odd $m$.

\subsection{The holonomic triangle in general}

The results of this subsection describe the dashed arrows of Figure~\ref{fig:triangle}. The first one shows that a moment of a D-finite weight satisfies a linear functional equation in the parameter.

\begin{mainprop}\label{prop:D}
Let $\varphi$ be holomorphic on a neighborhood of $[0,1)$ and of the form $\varphi(t)=(1-t)^{\alpha}\psi(t)$ near $t=1$, with $\alpha\in\C$, $\Re\alpha>-1$, $(1-t)^{\alpha}=\exp\bigl(\alpha\log(1-t)\bigr)$, and $\psi$ holomorphic at $1$. Then $h(z)=z\int_0^1t^{z-1}\varphi(t)\,\dd t$ extends from $\Re z>0$ to a meromorphic function on $\C$ with at most simple poles at the negative integers. It satisfies $h(0)=\varphi(0)$ and has polynomial growth on vertical strips in $\Re z>-1$. If moreover $\varphi$ is D-finite, then $h$ satisfies a homogeneous linear functional equation with polynomial coefficients.
\end{mainprop}

The last statement is a variant of the fact that the Mellin transform of a D-finite function satisfies such a functional equation \cite[Proposition~1 and the remark following it]{BorweinSalvy2008}. The second result gives a necessary condition for a family which satisfies a linear relation in the parameter. The condition concerns the way the parameter enters the factor $\Lambda(n,z)$ of the asymptotic behavior.

\begin{mainprop}\label{prop:E}
Let $u_n(z)$ satisfy a nonzero relation $\sum_{k=0}^{r}S_k(n,z)\,u_n(z+k)=0$ with $S_k\in\C(n,z)$, and assume that, uniformly on compact sets avoiding poles,
\[
u_n(z)=\Lambda(n,z)\bigl(h(z)+O(1/n)\bigr),\qquad \Lambda(n,z)=\lambda^n n^{az+b}\prod_j\Gamma(n+\alpha_jz+\beta_j)^{m_j},
\]
where $h$ is meromorphic on $\C$, $\lambda\in\C^{*}$ does not depend on $z$, $a,\alpha_j\in\R$, $b,\beta_j\in\C$, and $m_j\in\Z$. Put $\xi=a+\sum_jm_j\alpha_j$. If $\xi\notin\Q$ then $h\equiv0$. If $\xi\in\Q$ and $h\not\equiv0$, then $h$ satisfies a nonzero linear functional equation with coefficients in $\C(z)$ whose shifts lie in a single residue class modulo the denominator of $\xi$.
\end{mainprop}

For instance $w_n(z)=n!/(z)_n$ has $\xi=-1$ and gives $\Gamma(z+1)=z\Gamma(z)$. The family $(z/2)_n/n!$ has $\xi=1/2$ and gives $h(z+2)=(2/z)\,h(z)$ for $h=1/\Gamma(z/2)$. The family $(\sqrt2\,z)_n/n!$ has the nonzero connection constant $1/\Gamma(\sqrt2\,z)$ and $\xi=\sqrt2$, so it admits no relation in $z$ with coefficients in $\C(n,z)$. We call $\xi$ the \emph{slope} of $\Lambda(n,z)$. Proposition~\ref{prop:E} gives a necessary condition, and the example with $\xi=1/2$ shows that slopes need not be integers. Integer slopes, from which this work started, are a special case of the rationality condition. The condition $\xi\in\Q$ is an asymptotic analogue of the integrality condition in the theorem of Ore and Sato on multivariate hypergeometric terms \cite{AbramovPetkovsek2002}.

The upward arrows of Figure~\ref{fig:triangle} rest on a translation property. When the coefficients of a recurrence depend on $n$ and $z$ only through $n+z$, the shift $n\mapsto n+1$ is the shift $z\mapsto z+1$, and the connection constants inherit a functional equation in $z$.

\begin{mainthm}\label{thm:T}
Let $p_0,\ldots,p_r$ be rational functions of one variable, with $p_0p_r\ne0$. For $v\in\C^r$, let $u_n(z;v)$ be the solution of
\[
\sum_{k=0}^{r}p_k(n+z)\,u_{n+k}(z;v)=0\qquad(n\ge1)
\]
with $(u_1,\ldots,u_r)=v$, for the $z$ such that these coefficients are defined and $p_r(n+z)\ne0$ for all $n\ge1$, and let $T(z)$ be the matrix of the map $v\mapsto(u_2,\ldots,u_{r+1})$. Assume that there are a function $\Lambda(n,z)$ and a holomorphic function $\mathbf h$ on a half-plane, with values in row vectors, such that $u_n(z;v)/\Lambda(n,z)\to\mathbf h(z)\,v$ locally uniformly for every $v$, and that $\Lambda(n+1,z)/\Lambda(n,z+1)\to\eta(z)$ with $\eta\in\C(z)^{*}$. Then
\[
\mathbf h(z+1)\,T(z)=\eta(z)\,\mathbf h(z).
\]
Consequently, for every $v\in\C(z)^r$, the function $\mathbf h(z)\,v(z)$ extends meromorphically to $\C$ and satisfies a nonzero linear functional equation of order at most $r$ with coefficients in $\C(z)$.
\end{mainthm}

The dependence through $n+z$ is a sufficient condition, and Remark~\ref{rem:affine} shows that coefficients which are affine in $n$ and $z$ separately may give a connection function with no such equation. The existence of the limits is an analytic question. The next theorem answers it for the class where the coefficients have degree one in $n+z$, which contains the two worlds.

\begin{mainthm}\label{thm:F}
Let $r\ge1$, let $\alpha_0,\ldots,\alpha_r,\beta_0,\ldots,\beta_r\in\C$ with $\alpha_r\ne0$, and let $u_1,\ldots,u_r\in\C$. For $z$ such that $\alpha_r(n+z)+\beta_r\ne0$ for all $n\ge1$, define $u_n(z)$ by these initial values and
\begin{equation}\label{eq:classF}
\sum_{k=0}^{r}\bigl(\alpha_k(n+z)+\beta_k\bigr)\,u_{n+k}(z)=0\qquad(n\ge1).
\end{equation}
Put $P(x)=\sum_{k=0}^{r}\alpha_kx^{r-k}$, $Q(x)=\sum_{k=0}^{r}(\beta_k-k\alpha_k)\,x^{r-k}$, and $a=Q(0)/P(0)$. Assume that $P$ has a simple root $\rho>0$ such that $\gamma=Q(\rho)/\bigl(\rho P'(\rho)\bigr)$ satisfies $\Re\gamma>0$, and that every other root $\zeta$ of $P$ satisfies either $|\zeta|>\rho$, or $|\zeta|=\rho$ with $\zeta$ simple and $\Re\bigl(Q(\zeta)/(\zeta P'(\zeta))\bigr)<\Re\gamma$.
\begin{enumerate}[label=(\roman*)]
\item The generating function $U_z(x)=\sum_{n\ge1}u_n(z)\,x^n$ satisfies
\[
xP(x)\,U_z'(x)+\bigl(zP(x)+Q(x)\bigr)U_z(x)=R(x,z),
\]
where $R$ is a polynomial of degree at most $r$ in $x$ and at most one in $z$.
\item Let $\nu(x)=\exp\int_0^x\bigl(Q(t)-aP(t)\bigr)/\bigl(tP(t)\bigr)\,\dd t$ and $K=\lim(\rho-x)^{-\gamma}\nu(x)$ as $x\to\rho$ in $(0,\rho)$, with the principal branch of $(\rho-x)^{-\gamma}$. Then $K\ne0$, and for $\Re z>-1-\Re a$ we have $\rho^{n}n^{1-\gamma}u_n(z)\to h(z)$, locally uniformly in $z$, with
\[
h(z)=\frac{1}{K\,\Gamma(\gamma)\,\rho^{\gamma}}\int_0^1s^{z+a-1}\,\frac{\nu(\rho s)\,R(\rho s,z)}{P(\rho s)}\,\dd s .
\]
In particular $u_n(z)\sim h(z)\,\rho^{-n}n^{\gamma-1}$ when $h(z)\ne0$.
\item The function $h$ extends to a meromorphic function on $\C$ which satisfies a nonzero linear functional equation with polynomial coefficients.
\end{enumerate}
\end{mainthm}

The factor $\rho^{-n}n^{\gamma-1}$ does not depend on $z$, and $h$ is a combination of two moments over $[0,1]$ of hyperexponential functions. For the worlds of $\ee$ and $\pi$, one finds $\rho=1$ with $\gamma=2$ and $\gamma=3/2$, and (ii) gives back $f$ and $g$. After removing the first coefficients $\alpha_k(n+z)+\beta_k$ which vanish identically, which lowers the order and keeps the dependence on $n+z$, Theorem~\ref{thm:T} applies with $\Lambda(n,z)=\rho^{-n}n^{\gamma-1}$ and $\eta=\rho^{-1}$, so the functional equation of (iii) has order at most $r$. The reverse constructions start from a function of $z$.

\begin{mainthm}\label{thm:G}
Let $h$ be holomorphic in a half-plane $\Re z>c$. Assume that there are $c'>\max(c,0)$, $C>0$, and $\delta>0$ with $|h(z)|\le C|z|^{-\delta}$ for $\Re z\ge c'$, and that $h$ satisfies a nonzero equation $\sum_{k=0}^{r}R_k(z)\,h(z+k)=0$ with polynomial coefficients. Write $R_k(z)(z+k)=\sum_jc_{kj}z^j$ and $\theta=t\,\dd/\dd t$.
\begin{enumerate}[label=(\roman*)]
\item There is a unique continuous function $\varphi$ on $(0,1)$ such that $h(z)=z\int_0^1t^{z-1}\varphi(t)\,\dd t$ for $\Re z>c'$. It is given by $\varphi(t)=\frac{1}{2\pi i}\int_{c'-i\infty}^{c'+i\infty}h(z)\,t^{-z}\,\frac{\dd z}{z}$, and this integral vanishes for $t>1$.
\item The operator $L=\sum_{k,j}c_{kj}(-\theta)^jt^k$ is not zero, and $L\varphi=0$ in the sense of distributions on $(0,1)$. In particular $\varphi$ is real-analytic and D-finite on every open subinterval of $(0,1)$ on which the leading coefficient of $L$ does not vanish.
\end{enumerate}
\end{mainthm}

Theorem~\ref{thm:G}(i) needs decay. Polynomial growth is enough for a reconstruction by interpolation at the points $z_0,z_0+1,z_0+2,\ldots$.

\begin{mainthm}\label{thm:N}
Let $h$ be holomorphic in $\Re z>c$, and assume that $|h(z)|\le C(1+|z|)^M$ for $\Re z\ge\sigma$, where $\sigma>\max(c,0)$. Let $z_0>\sigma$ be real, and put
\[
\Delta_k=\sum_{j=0}^{k}(-1)^j\binom kj\,h(z_0+j),\qquad p_n(z)=\sum_{k=0}^{n}(-1)^k\binom{z-z_0}{k}\,\Delta_k .
\]
Then $p_n$ is the polynomial of degree at most $n$ which interpolates $h$ at $z_0,z_0+1,\ldots,z_0+n$, and $p_n(z)\to h(z)$ locally uniformly in $\Re z>\sigma$. If moreover $h$ satisfies a nonzero linear functional equation in integer shifts with polynomial coefficients, then $(p_n(z))_{n\ge0}$ is P-recursive with coefficients in $\C(z)$, and its generating function $\sum_{n\ge0}p_n(z)\,x^n$ is D-finite in $x$.
\end{mainthm}

\begin{corollary}\label{cor:closure}
Every connection function $h$ of Theorem~\ref{thm:F} is the locally uniform limit, on a right half-plane, of a family of polynomials $p_n\in\C[z]$ which is P-recursive in $n$ over $\C(z)$.
\end{corollary}

The interpolation formula is Newton's \cite{Norlund1926}, and Theorem~\ref{thm:N} adapts it to connection functions. The family $(p_n)$ has $\Lambda=1$, it is in general different from the family one starts from, and it need not belong to the class \eqref{eq:classF}. Theorem~\ref{thm:G}, applied to $h(z)/z^q$ for a large enough integer $q$, gives a D-finite weight for every $h$ of polynomial growth which satisfies a linear functional equation in integer shifts, as in the proof of Theorem~\ref{thm:N}.

The hypothesis of decay in Theorem~\ref{thm:G} may fail for a function of Theorem~\ref{thm:F}. The family $(n+z)\,u_{n+1}(z)=(n+z-\frac12)\,u_n(z)$, with $u_1=1$, belongs to the class \eqref{eq:classF} with $\rho=1$ and $\gamma=\frac12$, and $n^{1/2}u_n(z)\to\Gamma(z+1)/\Gamma(z+\frac12)$, a function which grows like $z^{1/2}$, although it has the weight $\pi^{-1/2}(1-t)^{-1/2}$. Corollary~\ref{cor:closure} closes the triangle for the whole class of Theorem~\ref{thm:F}. The functions $f$ and $g$ and the function of Example~\ref{ex:F} also satisfy the hypotheses of Proposition~\ref{prop:partial}, so that each of them is the limit of its partial sums as well. A functional equation alone does not determine $h$, since its solutions form a module over the periodic functions. For $f$ and $g$, Theorem~\ref{thm:B} shows that the functional equation and the growth condition determine the function up to a constant factor, hence also the weight, the partial sums, and the interpolation polynomials.

\subsection{Related work}

The passage between the differential equation of a weight and the functional equation of its moments is the Mellin correspondence. At the level of modules, it is the algebraic Mellin transform of Loeser and Sabbah, which turns a holonomic D-module on the torus into a difference module \cite{LoeserSabbah1991}. Garc\'{\i}a L\'opez describes the formal behavior at infinity of this difference module by a stationary phase formula \cite{GarciaLopez2023}. At the level of integer moments, Ablinger makes the correspondence algorithmic in both directions \cite{Ablinger2016,Ablinger2018}, and the first direction is also used by Borwein and Salvy \cite{BorweinSalvy2008}. For constants, Huang and Kauers study the limits of P-recursive sequences as D-finite numbers and relate them to the values of D-finite functions \cite{HuangKauers2018}. We do not know of a previous formulation, for a class of recurrences, of the passage between the recurrence in $n$ and the functional equation of the connection constant in $z$, as in Theorems~\ref{thm:T}, \ref{thm:F}, \ref{thm:G}, and~\ref{thm:N}. Newton's interpolation series, used in Theorem~\ref{thm:N}, are classical \cite{Norlund1926}. The formal theory leaves open the selection of one solution among the solutions of the functional equation. For the two equations of this paper, the growth condition of Theorem~\ref{thm:B} makes this selection, while Theorem~\ref{thm:G} reconstructs a weight from a given solution.

The closest generalization of Wielandt's theorem that we know is due to Pérez-Marco. For the equations $sF(s)=\sum_{k=1}^{d}\alpha_kF(s+k)$ with constant $\alpha_k$, he proves that the meromorphic solutions satisfying a one-sided exponential bound on a vertical strip form a space of dimension $d$, spanned by integrals $\int_0^{+\infty\cdot\omega_k}t^{s-1}\ee^{P(t)}\,\dd t$ attached to a polynomial $P$ \cite[Theorem~1.1 and the theorem on the basis of Omega functions in Section~6]{PerezMarco2025}. Equations \eqref{eq:fe-f} and \eqref{eq:fe-g} are not of this form. Their solutions are integrals over $[0,1]$ and not along a path to infinity, the weight of the world of $\pi$ is algebraic, and the growth condition of Theorem~\ref{thm:B} is two-sided. The functions $f$ and $g$ admit representations by hypergeometric functions whose parameters depend on $z$ (Remark~\ref{rem:hyp}), and constants which depend on parameters often appear in the asymptotics of hypergeometric sequences, Euler's Gamma function being the first example. The contribution of this paper is the translation property of Theorem~\ref{thm:T}, the construction of Theorem~\ref{thm:F} for recurrences which are affine in $n+z$, the reconstructions of Theorems~\ref{thm:G} and~\ref{thm:N}, the characterization of Theorem~\ref{thm:B}, and the arithmetic interpretation of Theorem~\ref{thm:C}, with the worlds of $\ee$ and $\pi$ as guiding examples and the holonomic triangle as the frame which relates them.

\subsection{Organization}

The paper follows the order of the research. Sections~\ref{sec:moments} to~\ref{sec:paths} prove the results on the two worlds, Theorem~\ref{thm:A} in Section~\ref{sec:moments}, Theorem~\ref{thm:B} and the differential transcendence in Section~\ref{sec:wielandt}, Theorem~\ref{thm:C} in Section~\ref{sec:trace}, and the comparison of the two worlds in Section~\ref{sec:paths}. Section~\ref{sec:branch} turns to the general triangle and proves Propositions~\ref{prop:D} and~\ref{prop:E} and Theorems~\ref{thm:T} to~\ref{thm:N}. Section~\ref{sec:open} lists open questions, and Section~\ref{sec:oeis} lists the sequences of the OEIS which occur in the paper.

\section{Connection functions as moments}\label{sec:moments}

This section proves Theorem~\ref{thm:A}. We write
\[
\varphi_{\ee}(t)=(1-t)\,\ee^{t},\qquad \varphi_{\pi}(t)=(1-t)^{1/2}(1+t)^{-1/2},
\]
so that $h_{\omega_\ee}(z)=z\int_0^1t^{z-1}\varphi_\ee(t)\,\dd t$ and $h_{\omega_\pi}(z)=z\int_0^1t^{z-1}\varphi_\pi(t)\,\dd t$ for $\Re z>0$. The recurrences \eqref{eq:param} make sense for every complex value of the parameter that is not a negative integer, and we write $u_n(z)$, $v_n(z)$ for the sequences obtained with $m$ replaced by $z$. For $\varphi\in\{\varphi_\ee,\varphi_\pi\}$ and $|x|<1$ put
\[
F_\varphi(z,x)=z\int_0^1s^{z-1}\varphi(xs)\,\dd s\quad(\Re z>0),\qquad F_\varphi(0,x)=\varphi(0)=1 .
\]
The function $F_\varphi(z,x)$ is analytic in $x$ for $x\notin(-\infty,-1]\cup[1,\infty)$, it tends to $F_\varphi(0,x)$ as $z\to0^{+}$, and $F_\varphi(z,1)=h_{\omega}(z)$ for the corresponding weight.

\subsection{Generating functions}

We first compute the generating functions of the sequences in closed form.

\begin{lemma}\label{lem:gf}
Let $\Re z>0$ or $z=0$, and let $U_z(x)=\sum_{n\ge1}u_n(z)x^n$ and $V_z(x)=\sum_{n\ge1}v_n(z)x^n$. For $|x|<1$,
\[
U_z(x)=\frac{x^2\,\ee^{-x}}{(1-x)^2}\,F_{\varphi_\ee}(z,x),\qquad V_z(x)=\frac{x^2}{(1-x)^{3/2}(1+x)^{1/2}}\,F_{\varphi_\pi}(z,x).
\]
\end{lemma}

\begin{proof}
Write the recurrences as $(n+z)u_{n+2}=(n+z)u_{n+1}+u_n$ and $(n+z)v_{n+2}=v_{n+1}+(n+z)v_n$, multiply by $x^{n+1}$ and sum over $n\ge1$. With $u_1=v_1=0$ and $u_2=v_2=1$ this gives
\[
(1-x)\,U_z'=\frac{x^2+(z-1)\,x+2-z}{x}\,U_z+zx,\qquad (1-x^2)\,V_z'=\frac{z\,x^2+x+2-z}{x}\,V_z+zx .
\]
Conversely, comparing coefficients shows that a power series $\sum_{n\ge1}c_nx^n$ with $c_1=0$ and $c_2=1$ satisfying either equation has coefficients satisfying the corresponding recurrence, so each equation has exactly one such solution. These inhomogeneous equations have order one because the coefficients of the recurrences have degree one in $n$, and the operator $x\,\dd/\dd x$ replaces multiplication by $n$. For the first equation, the partial fraction decomposition
\[
\frac{x^2+(z-1)\,x+2-z}{x(1-x)}=-1+\frac{2-z}{x}+\frac{2}{1-x}
\]
shows that $\mu(x)=x^{z-2}(1-x)^2\ee^{x}$ is an integrating factor, with $(\mu U_z)'=z\,x^{z-1}\varphi_\ee(x)$. For the second, the decomposition $\frac{2-z}{x}+\frac{3/2}{1-x}-\frac{1/2}{1+x}$ gives the integrating factor $\mu(x)=x^{z-2}(1-x)^{3/2}(1+x)^{1/2}$ and $(\mu V_z)'=z\,x^{z-1}\varphi_\pi(x)$. When $\Re z>0$, we integrate from $0$ and substitute $t=xs$, so that $\mu(x)^{-1}z\int_0^xt^{z-1}\varphi(t)\,\dd t$ becomes the right-hand side of the lemma. This right-hand side is a power series in $x$ of the form $x^2+O(x^3)$ which solves the differential equation, hence it equals the generating function. When $z=0$, the same holds with $F_\varphi(0,x)=1$, because $\mu^{-1}$ is then itself a solution of the form $x^2+O(x^3)$.
\end{proof}

\subsection{Asymptotics of the sequences}

The closed forms of Lemma~\ref{lem:gf} give the asymptotic behavior of the sequences, also for complex values of the parameter.

\begin{proposition}\label{prop:asym}
Let $\Re z>0$ or $z=0$. Then $u_n(z)/n\to f(z)$ and $v_n(z)/\sqrt n\to g(z)$ as $n\to\infty$, where
\[
\ee f(z)=F_{\varphi_\ee}(z,1)=h_{\omega_\ee}(z),\qquad \sqrt{\pi/2}\;g(z)=F_{\varphi_\pi}(z,1)=h_{\omega_\pi}(z),
\]
the value at $z=0$ being $1$ in both cases. For real $z\ge0$ these limits are positive, so that $u_n(z)\sim f(z)\,n$ and $v_n(z)\sim g(z)\sqrt n$.
\end{proposition}

\begin{proof}
Since $\varphi_\ee$ is entire, $\Phi(x)=x^2\ee^{-x}F_{\varphi_\ee}(z,x)$ is an entire function of $x$, and $U_z(x)=\Phi(x)(1-x)^{-2}$. Writing $\Phi(x)=\Phi(1)-\Phi'(1)(1-x)+(1-x)^2R(x)$ with $R$ entire gives $u_n(z)=\Phi(1)(n+1)-\Phi'(1)+[x^n]R(x)$, and $[x^n]R(x)\to0$. Hence $u_n(z)\sim\Phi(1)\,n$ with $\Phi(1)=\ee^{-1}F_{\varphi_\ee}(z,1)$.

The function $V_z$ is analytic in $\C\setminus\bigl((-\infty,-1]\cup[1,\infty)\bigr)$, because $xs$ stays on the segment $[0,x]$ in the integral defining $F_{\varphi_\pi}(z,x)$. Near $x=1$, we have $V_z(x)=(1-x)^{-3/2}\Psi(x)$ with $\Psi(x)=x^2(1+x)^{-1/2}F_{\varphi_\pi}(z,x)$, and $\Psi(x)=\Psi(1)+O(|1-x|)$ in a slit neighborhood of $1$, because the derivative $z\int_0^1s^{z}\varphi_\pi'(xs)\,\dd s$ stays bounded there, $\varphi_\pi'(t)$ being $O\bigl((1-t)^{-1/2}\bigr)$. Near $x=-1$, the function $V_z$ is $O\bigl((1+x)^{-1/2}\bigr)$. The transfer theorems of singularity analysis, applied at the two singularities $\pm1$ \cite[Theorems~VI.1 and~VI.3, Section~VI.5]{FlajoletSedgewick2009}, give $v_n(z)=\Psi(1)\,n^{1/2}/\Gamma(3/2)+O(n^{-1/2})$, and $\Psi(1)/\Gamma(3/2)=2^{-1/2}F_{\varphi_\pi}(z,1)\cdot2/\sqrt\pi=\sqrt{2/\pi}\,F_{\varphi_\pi}(z,1)$.
\end{proof}

For $z=m\in\N$, Proposition~\ref{prop:asym} proves the first statement of Theorem~\ref{thm:A}. The proposition holds for complex values of the parameter as well, so $f$ and $g$ are, on $\Re z>0$, the connection constants of the recurrences with parameter $z$.

\subsection{Closed forms and functional equations}

For $\Re s>-1$, put $I(s)=\int_0^1t^{s}\ee^{t}\,\dd t$ and $J(s)=\int_0^1t^{s}(1-t^2)^{-1/2}\,\dd t$. Expanding $\ee^{t}$ gives $I(s)=\sum_{k\ge0}1/\bigl(k!\,(s+k+1)\bigr)$, and the substitution $u=t^2$ gives
\[
J(s)=\tfrac12\,B\!\left(\tfrac{s+1}{2},\tfrac12\right)=\frac{\sqrt\pi}{2}\,\frac{\Gamma\bigl(\frac{s+1}{2}\bigr)}{\Gamma\bigl(\frac s2+1\bigr)},
\]
where $B$ is Euler's Beta function.
Integration by parts gives $I(s+1)=\ee-(s+1)I(s)$. By definition \eqref{eq:operation}, $h_{\omega_\ee}(z)=z\bigl(I(z-1)-I(z)\bigr)$ and $h_{\omega_\pi}(z)=z\bigl(J(z-1)-J(z)\bigr)$.

\begin{proof}[Proof of Theorem~\ref{thm:A}]
From $I(z)=\ee-zI(z-1)$ and $I(z+1)=\ee-(z+1)I(z)$ we obtain $z\bigl(I(z-1)-I(z)\bigr)=\ee-(z+1)I(z)=I(z+1)$, hence
\[
\ee f(z)=I(z+1)=\sum_{k\ge0}\frac{1}{k!\,(z+k+2)} .
\]
The series converges locally uniformly away from the poles and gives the meromorphic continuation of $f$, with a simple pole of residue $1/(\ee\,k!)$ at $z=-k-2$. The relation $I(z+2)=\ee-(z+2)I(z+1)$, valid on $\C$ by continuation, reads \eqref{eq:fe-f1}. Subtracting this relation from the same relation at $z+1$ gives $f(z+2)=(z+2)\bigl(f(z)-f(z+1)\bigr)$, which is \eqref{eq:fe-f}.

For $g$, the formula for $J$ and the functional equation of $\Gamma$ give $zJ(z-1)=\sqrt\pi\,A(z)$ and $zJ(z)=\sqrt\pi\,A(z-1)$, with $A$ as in \eqref{eq:closed}, so that $\sqrt{\pi/2}\,g(z)=\sqrt\pi\bigl(A(z)-A(z-1)\bigr)$, which is \eqref{eq:closed}. The function $A$ has simple poles at $z=-2,-4,\ldots$ and zeros at $z=-1,-3,\ldots$, and $A(z-1)$ has simple poles at $z=-1,-3,\ldots$ and zeros at $z=0,-2,\ldots$. The poles of the two terms are therefore not cancelled, and $g$ has simple poles at the negative integers and is holomorphic at $0$. The functional equation of $\Gamma$ gives $A(z)A(z-1)=\Gamma(\frac z2+1)/\Gamma(\frac z2)=z/2$, hence $A(z+1)=(z+1)/\bigl(2A(z)\bigr)$ and $A(z+2)=\frac{z+2}{z+1}A(z)$. Inserting these relations into
\[
(z+1)\bigl(A(z+2)-A(z+1)\bigr)-(z+1)\bigl(A(z)-A(z-1)\bigr)+\bigl(A(z+1)-A(z)\bigr)
\]
reduces it to $-zA(z+1)+(z+1)A(z-1)=-\frac{z(z+1)}{2A(z)}+\frac{(z+1)z}{2A(z)}=0$, which is \eqref{eq:fe-g}.

In a strip $a\le\Re z\le b$ with $a>-2$, we have $|z+k+2|\ge a+k+2>0$, so the series for $\ee f$ is bounded there, and the same holds in every strip for $|\tau|\ge1$, since then $|z+k+2|\ge1$. If $|\arg z|\le\pi-\varepsilon$, then $|z+k+2|\ge\sin(\varepsilon/2)\,(|z|+k+2)$, so $|z/(z+k+2)|\le1/\sin(\varepsilon/2)$, and dominated convergence gives $z\,f(z)\to\ee^{-1}\sum_{k\ge0}1/k!=1$ as $|z|\to\infty$ in this sector. For $g$ we use the expansion $\Gamma(w+\alpha)/\Gamma(w+\beta)=w^{\alpha-\beta}\bigl(1+\frac{(\alpha-\beta)(\alpha+\beta-1)}{2w}+O(w^{-2})\bigr)$, valid as $|w|\to\infty$ in $|\arg w|\le\pi-\delta$ \cite[5.11.13]{DLMF}. With $w=z/2$ it gives $A(z)=w^{1/2}\bigl(1+\frac1{8w}+O(w^{-2})\bigr)$ and $A(z-1)=w^{1/2}\bigl(1-\frac1{8w}+O(w^{-2})\bigr)$, hence $g(z)=\sqrt2\,w^{1/2}\bigl(\frac1{4w}+O(w^{-2})\bigr)\sim1/(2\sqrt z)$ as $|z|\to\infty$ in $|\arg z|\le\pi-\varepsilon$. This expansion bounds $g$ on every vertical strip for $|\tau|$ large, and $g$ is continuous on the remaining compact parts of $\{|\tau|\ge1\}$ and of the strips $a\le\Re z\le b$ with $a>-1$, where it is holomorphic.
\end{proof}

We end this section with the hypergeometric form of the connection functions.

\begin{remark}\label{rem:hyp}
The connection functions admit representations by hypergeometric functions whose parameters depend on $z$. Let ${}_1F_1$ be the confluent hypergeometric function and ${}_2F_1$ the Gauss hypergeometric function. Kummer's transformation and Euler's integral \cite[13.2.39 and 15.6.1]{DLMF} give
\[
\ee f(z)=\frac{\ee}{z+2}\,{}_1F_1(1;z+3;-1),\qquad \sqrt{\pi/2}\;g(z)=z\,B\bigl(z,\tfrac32\bigr)\,{}_2F_1\bigl(\tfrac12,z;z+\tfrac32;-1\bigr),
\]
and \eqref{eq:fe-f}, \eqref{eq:fe-g} can also be derived from contiguous relations. We have preferred the moment representation, which treats the two worlds in the same way and leads directly to the discarded solutions of Section~\ref{sec:wielandt}.
\end{remark}

\section{Uniqueness and transcendence of the connection functions}\label{sec:wielandt}

Throughout this section, we use the functions
\[
H(z)=\ee^{i\pi z}\,\Gamma(z+2),\qquad \tilde g(z)=A(z)+A(z-1),\qquad G(z)=\ee^{i\pi z}\,\tilde g(z).
\]
The function $\tilde g$ is holomorphic on $\Re z>-1$, since the poles of $A(z)$ lie at $z=-2,-4,\ldots$ and those of $A(z-1)$ at $z=-1,-3,\ldots$. We use two estimates on a fixed vertical strip. By Stirling's formula \cite[5.11.9]{DLMF},
\[
|\Gamma(z+2)|=\sqrt{2\pi}\,|\tau|^{\Re z+3/2}\,\ee^{-\pi|\tau|/2}\bigl(1+o(1)\bigr)\qquad(|\tau|\to\infty),
\]
and $|\ee^{i\pi z}|=\ee^{-\pi\tau}$. By the expansion used at the end of Section~\ref{sec:moments}, $\tilde g(z)\sim\sqrt{2z}$ as $|\tau|\to\infty$.

\subsection{Periodic functions}

The proof of Theorem~\ref{thm:B} reduces to the following lemma on periodic entire functions.

\begin{lemma}\label{lem:periodic}
Let $p$ be an entire function with $p(z+1)=p(z)$. Assume that on a vertical strip of width at least $1$ there are constants $C,N$ and real numbers $\beta_+<2\pi$ and $\beta_->-2\pi$ such that
\[
|p(z)|\le C(1+|\tau|)^N\ee^{\beta_+\tau}\quad(\tau\ge1),\qquad |p(z)|\le C(1+|\tau|)^N\ee^{-\beta_-|\tau|}\quad(\tau\le-1).
\]
Then $p$ is constant, and $p=0$ if $\beta_->0$.
\end{lemma}

\begin{proof}
Write $p(z)=P(q)$ with $q=\ee^{2\pi iz}$ and $P$ holomorphic on $\C^{*}$. The strip contains a fundamental domain of the translation $z\mapsto z+1$, so the bounds describe $P$ near $q=0$, which corresponds to $\tau\to+\infty$, and near $q=\infty$, which corresponds to $\tau\to-\infty$, while $P$ is bounded on the annulus corresponding to $|\tau|\le1$. Since $|q|=\ee^{-2\pi\tau}$, the first bound gives $|P(q)|=O\bigl(|\log|q||^{N}\,|q|^{-\beta_+/(2\pi)}\bigr)$ as $q\to0$, and $\beta_+/(2\pi)<1$ implies $qP(q)\to0$, so $0$ is a removable singularity. The second bound gives $|P(q)|=O\bigl((\log|q|)^{N}|q|^{-\beta_-/(2\pi)}\bigr)=o(|q|)$ as $q\to\infty$, because $\beta_->-2\pi$. By the Cauchy estimates, the Laurent coefficients of $P$ of positive index vanish, so $P$ is constant, and the constant is $0$ when $\beta_->0$ because $P(q)\to0$ as $q\to\infty$.
\end{proof}

\subsection{The discarded solutions}

The next lemma gives the equations and the Casoratians of the discarded solutions.

\begin{lemma}\label{lem:discarded}
The function $H$ satisfies $H(z+1)=-(z+2)H(z)$, hence \eqref{eq:fe-f}, and $\mathcal{C}(f,H)=-H$. The function $G$ satisfies \eqref{eq:fe-g}, and $\mathcal{C}(g,G)=-\sqrt2\,\ee^{i\pi z}$. If $p$ is $1$-periodic, then $pH$ and $pG$ satisfy the same equations as $H$ and $G$.
\end{lemma}

\begin{proof}
The relation for $H$ follows from $\Gamma(z+3)=(z+2)\Gamma(z+2)$ and $\ee^{i\pi(z+1)}=-\ee^{i\pi z}$. Any solution $y_0$ of $y_0(z+1)=-(z+2)y_0(z)$ satisfies $y_0(z+2)=(z+3)(z+2)y_0(z)=(z+2)\bigl(y_0(z)-y_0(z+1)\bigr)$. Using $f(z+1)=1-(z+2)f(z)$, we obtain
\[
\mathcal{C}(f,H)(z)=f(z)H(z+1)-f(z+1)H(z)=-H(z)\bigl((z+2)f(z)+f(z+1)\bigr)=-H(z).
\]
For $G$, the relations $A(z-1)=z/\bigl(2A(z)\bigr)$, $A(z+1)=(z+1)/\bigl(2A(z)\bigr)$, and $A(z+2)=\frac{z+2}{z+1}A(z)$ of Section~\ref{sec:moments} give, with $A=A(z)$,
\[
(z+1)\tilde g(z+2)-(z+1)\tilde g(z)-\tilde g(z+1)=\Bigl((z+2)-(z+1)-1\Bigr)A+z\,A(z+1)-(z+1)A(z-1)=0,
\]
and since $\ee^{i\pi(z+1)}=-\ee^{i\pi z}$ this is \eqref{eq:fe-g} for $G$. For the Casoratian,
\[
\mathcal{C}(g,G)(z)=-\ee^{i\pi z}\bigl(g(z)\tilde g(z+1)+g(z+1)\tilde g(z)\bigr),
\]
and inserting $g=\sqrt2\,\bigl(A-\frac z{2A}\bigr)$, $\tilde g=A+\frac z{2A}$, $g(z+1)=\sqrt2\,\bigl(\frac{z+1}{2A}-A\bigr)$, and $\tilde g(z+1)=\frac{z+1}{2A}+A$, the bracket equals $\sqrt2\bigl(\tfrac12+\tfrac12\bigr)=\sqrt2$, the terms in $A^2$ and in $z(z+1)/(4A^2)$ cancelling. The last statement holds because the coefficients of the equations are unchanged by $z\mapsto z+1$ applied to $p$.
\end{proof}

The Casoratian of two solutions $y_1,y_2$ of \eqref{eq:fe-g} satisfies $\mathcal{C}(y_1,y_2)(z+1)=-\mathcal{C}(y_1,y_2)(z)$. To see this, we replace $y_j(z+2)$ by $y_j(z)-y_j(z+1)/(z+1)$ in $\mathcal{C}(y_1,y_2)(z+1)=y_1(z+1)y_2(z+2)-y_1(z+2)y_2(z+1)$. The terms in $1/(z+1)$ cancel, and what remains is $y_1(z+1)y_2(z)-y_1(z)y_2(z+1)$.

\subsection{Proof of Theorem~\texorpdfstring{\ref{thm:B}}{1.2}}

\begin{proof}[Proof of Theorem~\ref{thm:B}]
Let $y$ be meromorphic on $\C$, holomorphic and of polynomial growth on vertical strips in $\Re z>c$. We may assume $c\ge-1$, and we fix a strip $a\le\Re z\le a+1$ with $a>c$.

In case (i), put $k(z)=y(z+1)+(z+2)y(z)$. Then
\[
k(z+1)=y(z+2)+(z+3)y(z+1)=(z+2)y(z)-(z+2)y(z+1)+(z+3)y(z+1)=k(z),
\]
so $k$ is $1$-periodic and holomorphic on $\Re z>c$, hence entire. It has polynomial growth in this strip, so Lemma~\ref{lem:periodic} with $\beta_+=\beta_-=0$ shows that $k$ is a constant $\lambda$. The function $y_0=y-\lambda f$ then satisfies $y_0(z+1)=-(z+2)y_0(z)$, because $f(z+1)+(z+2)f(z)=1$. The quotient $p=y_0/H$ is $1$-periodic by Lemma~\ref{lem:discarded}. It is holomorphic on $\Re z>c$, because $\Gamma$ has no zeros, hence entire. In the strip $a\le\Re z\le a+1$, the function $y_0$ has polynomial growth and $|1/H(z)|\le C(1+|\tau|)^{N'}\ee^{\pi\tau+\pi|\tau|/2}$. This is $\ee^{3\pi\tau/2}$ up to a polynomial factor for $\tau\ge1$ and $\ee^{-\pi|\tau|/2}$ for $\tau\le-1$. Lemma~\ref{lem:periodic} with $\beta_+=3\pi/2$ and $\beta_-=\pi/2$ gives $p=0$, so $y=\lambda f$.

In case (ii), the function $p=\mathcal{C}(g,y)/\mathcal{C}(g,G)$ is $1$-periodic, since both Casoratians change sign under $z\mapsto z+1$. It is holomorphic on $\Re z>c$, because $\mathcal{C}(g,G)=-\sqrt2\,\ee^{i\pi z}$ has no zeros and $g$ is holomorphic on $\Re z>-1$, hence entire. In the strip $a\le\Re z\le a+1$, the numerator has polynomial growth and $|1/\mathcal{C}(g,G)|=\ee^{\pi\tau}/\sqrt2$, so Lemma~\ref{lem:periodic} with $\beta_+=\pi$ and $\beta_-=\pi$ gives $p=0$, that is $g(z)y(z+1)=g(z+1)y(z)$. Since $g(z)\sim1/(2\sqrt z)$ as $|z|\to\infty$ with $|\arg z|\le\pi/2$, there is $c_1$ such that $g$ has no zeros on $\Re z\ge c_1$. On $\Re z>\max(c_1,c)$ the quotient $r=y/g$ is holomorphic and $1$-periodic, hence extends to an entire function, and $|r(z)|\le C(1+|\tau|)^{N}\cdot3|z|^{1/2}$ on a strip of width $1$ in this half-plane. Lemma~\ref{lem:periodic} with $\beta_+=\beta_-=0$ shows that $r$ is a constant $\lambda$, so $y=\lambda g$ on a half-plane, hence on $\C$.
\end{proof}

\subsection{The hypotheses of Theorem~\texorpdfstring{\ref{thm:B}}{1.2}}

The next proposition proves the four examples given after Theorem~\ref{thm:B}.

\begin{proposition}\label{prop:sharp}
The four functions
\[
f(z)+\frac{\Gamma(z+2)}{\sin\pi z},\qquad g(z)+\frac{\tilde g(z)}{\sin\pi z},\qquad f(z)+\frac{1}{\Gamma(-1-z)},\qquad g(z)+\tilde g(z)\sin\pi z
\]
satisfy \eqref{eq:fe-f}, \eqref{eq:fe-g}, \eqref{eq:fe-f}, and \eqref{eq:fe-g}, respectively. The first two have polynomial growth on vertical strips for $|\tau|\ge1$ and poles at all integers $m\ge-1$, respectively $m\ge0$. The third is holomorphic on $\Re z>-2$, coincides with $f$ on $\{-1,0,1,\ldots\}$, and has exponential type $\pi/2$ on vertical strips. The fourth is holomorphic on $\Re z>-1$, coincides with $g$ on $\N$, and has exponential type $\pi$ on vertical strips.
\end{proposition}

\begin{proof}
Each added term is the product of $H$ or $G$ by a $1$-periodic function, so Lemma~\ref{lem:discarded} gives the equations. We have $\Gamma(z+2)/\sin\pi z=H\cdot\ee^{-i\pi z}/\sin\pi z$ and $\tilde g/\sin\pi z=G\cdot\ee^{-i\pi z}/\sin\pi z$, while the reflection formula $\Gamma(z+2)\Gamma(-1-z)=\pi/\sin\pi z$ gives $1/\Gamma(-1-z)=H\cdot\ee^{-i\pi z}\sin(\pi z)/\pi$, and $\tilde g\sin\pi z=G\cdot\ee^{-i\pi z}\sin\pi z$. The functions $\ee^{-i\pi z}/\sin\pi z$ and $\ee^{-i\pi z}\sin\pi z$ are $1$-periodic since $\sin\pi(z+1)=-\sin\pi z$.

The growth statements follow from $\frac13\ee^{\pi|\tau|}\le|\sin\pi z|\le\ee^{\pi|\tau|}$ for $|\tau|\ge1$, from the estimate on $\Gamma(z+2)$, and from $\tilde g(z)\sim\sqrt{2z}$. The added terms decay like $\ee^{-3\pi|\tau|/2}$ and $\ee^{-\pi|\tau|}$ in the first two cases, and grow like $\ee^{\pi|\tau|/2}$ and $\ee^{\pi|\tau|}$, up to polynomial factors, in the last two. The term $\Gamma(z+2)/\sin\pi z$ has a pole at each integer $m\ge-1$, where $\Gamma(m+2)\ne0$. The term $\tilde g/\sin\pi z$ has a pole at each $m\in\N$, where $\tilde g(m)>0$ because $A(m)>0$ and $A(m-1)\ge0$. The entire function $1/\Gamma(-1-z)$ vanishes at $z=-1,0,1,\ldots$. The function $\tilde g\sin\pi z$ is entire, because the simple poles of $\tilde g$ at the negative integers are cancelled by the zeros of $\sin\pi z$, and it vanishes on $\N$.
\end{proof}

The last two examples are optimal in the following sense.

\begin{proposition}\label{prop:types}
In Theorem~\ref{thm:B}, the assumption of polynomial growth can be replaced by a bound $|y(z)|\le C(1+|\tau|)^N\ee^{\alpha|\tau|}$ on vertical strips in the half-plane, with $\alpha<\pi/2$ in case~\textup{(i)} and $\alpha<\pi$ in case~\textup{(ii)}. By Proposition~\ref{prop:sharp}, neither bound can be relaxed to $\alpha=\pi/2$, respectively $\alpha=\pi$.
\end{proposition}

\begin{proof}
We follow the proof of Theorem~\ref{thm:B}, using Lemma~\ref{lem:periodic} with negative values of $\beta_-$ allowed. In case~(i), the periodic function $k$ satisfies the bounds of the lemma with $\beta_+=\alpha$ and $\beta_-=-\alpha$, hence is constant. The quotient $p=y_0/H$ then satisfies them with $\beta_+=\alpha+3\pi/2<2\pi$ and $\beta_-=\pi/2-\alpha>0$, so $p=0$. In case~(ii), $p=\mathcal{C}(g,y)/\mathcal{C}(g,G)$ satisfies them with $\beta_+=\pi+\alpha<2\pi$ and $\beta_-=\pi-\alpha>0$, so $p=0$, and $r=y/g$ satisfies them with $\beta_+=\alpha$ and $\beta_-=-\alpha$, so $r$ is constant.
\end{proof}

The four examples show that Theorem~\ref{thm:B} needs both of its hypotheses. Without holomorphy in a half-plane, polynomial growth does not exclude the discarded solutions multiplied by $1/\sin\pi z$. Without a growth condition, agreement on the integers does not determine the solution, and Proposition~\ref{prop:types} locates the threshold, at exponential type $\pi/2$ for $f$ and $\pi$ for $g$.

\subsection{Algebraic differential equations}

Euler's Gamma function satisfies no algebraic differential equation, by a theorem of H\"older. The same holds for the two connection functions.

\begin{proposition}\label{prop:hypertr}
The functions $f$ and $g$ are differentially transcendental over $\C(z)$, that is, there is no nonzero polynomial $P$ with $P\bigl(z,y,y',\ldots,y^{(k)}\bigr)=0$ for $y=f$ or $y=g$.
\end{proposition}

\begin{proof}
Let $\mathcal F$ be the field of $1$-periodic meromorphic functions. We use two results of Hardouin and Singer \cite{HardouinSinger2008}. If $b\in\C(x)$ and $u\ne0$ is meromorphic on $\C$ with $u(x+1)=b(x)u(x)$, then $u$ is differentially algebraic over $\mathcal F(x)$ only if $b=c\,r(x+1)/r(x)$ with $c\in\C$ and $r\in\C(x)$ \cite[Corollary~3.4(1)]{HardouinSinger2008}. If $a,b\in\C(x)$, where $a$ is not constant and no two zeros or poles of $a$ differ by a nonzero integer, and if $y$ is meromorphic on $\C$ with $y(x+1)=a(x)y(x)+b(x)$ and $y\notin\mathcal F(x)$, then $y$ is not differentially algebraic over $\mathcal F(x)$ \cite[Proposition~3.9(1)]{HardouinSinger2008}.

For $f$ we have $f(z+1)=-(z+2)f(z)+1$, and $a=-(z+2)$ satisfies the condition. Suppose that $f\in\mathcal F(z)$. For a generic point $z_0$, the periodic coefficients are constant on $z_0+\Z$, so $f(z_0+n)=R(n)$ for a rational function $R$, and $R(X+1)=1-(X+z_0+2)R(X)$ for all integers $X\ge0$, hence identically. A pole of $R$ with minimal real part $X_0$ would give the pole $X_0-1$ of $R(X+1)$, which is not a pole of the right-hand side, so $R$ is a polynomial, and the degrees of the two sides differ. Hence $f\notin\mathcal F(z)$, and $f$ is differentially transcendental over $\mathcal F(z)$, hence over $\C(z)$.

For $g$, put $\Upsilon(x)=A(2x)=\Gamma(x+1)/\Gamma(x+\frac12)$, so that $\Upsilon(x+1)=\frac{x+1}{x+1/2}\,\Upsilon(x)$. The divisor of $r(x+1)/r(x)$ has degree $0$ on each orbit $x_0+\Z$, while $(x+1)/(x+\frac12)$ has a zero and a pole in different orbits, so $\Upsilon$, hence $A$, is differentially transcendental. Since $A(z)A(z-1)=\Gamma(\frac z2+1)/\Gamma(\frac z2)=\frac z2$, Theorem~\ref{thm:A} gives $g=\sqrt2\,\bigl(A-\frac{z}{2A}\bigr)$, that is $A^2-\frac{g}{\sqrt2}A-\frac z2=0$. If $g$ were differentially algebraic over $\C(z)$, so would be $A$, which is algebraic over $\C(z,g)$.
\end{proof}

\section{The linear forms at the integers}\label{sec:trace}

This section proves Theorem~\ref{thm:C}, and with it the formulas \eqref{eq:arith}, \eqref{eq:ab}, and \eqref{eq:pq}. We keep the notation $H$, $G$, and $\tilde g$ of Section~\ref{sec:wielandt}.

\begin{proof}[Proof of Theorem~\ref{thm:C}~\textup{(i)}]
By Theorem~\ref{thm:A}, $\ee f$ has a simple pole with residue $1/k!$ at $z=-k-2$ for every $k\in\N$. At the same point $\Gamma(z+2)$ has residue $(-1)^k/k!$ and $\ee^{i\pi z}=(-1)^k$, so $H$ has residue $1/k!$ as well. Both functions are holomorphic elsewhere, hence $E=H-\ee f$ is entire. Lemma~\ref{lem:discarded} gives $H(z+1)+(z+2)H(z)=0$, and $f(z+1)+(z+2)f(z)=1$, so
\[
E(z+1)+(z+2)E(z)=-\ee .
\]
Taking the difference of this relation at $z$ and $z+1$ shows that $E$ satisfies \eqref{eq:fe-f}, as $H$ does. We have $H(m)=(-1)^m(m+1)!=(-1)^ma_m$. For $E$, we use induction on $m$. We have $E(0)=H(0)-\ee f(0)=1-1=0=\ee D_1$. If $E(m)=(-1)^m\ee D_{m+1}$, then
\[
E(m+1)=-\ee-(m+2)E(m)=(-1)^{m+1}\ee\bigl((m+2)D_{m+1}+(-1)^{m}\bigr)=(-1)^{m+1}\ee D_{m+2},
\]
by the recurrence $D_{n+1}=(n+1)D_n+(-1)^{n+1}$ of the derangement numbers. Therefore
\[
(-1)^m\ee f(m)=(-1)^m\bigl(H(m)-E(m)\bigr)=a_m-\ee b_m,
\]
which is the first identity in \eqref{eq:arith} with the coefficients \eqref{eq:ab}. The modulus $|H(z)|=\ee^{-\pi\tau}|\Gamma(z+2)|$ grows like $\ee^{\pi|\tau|/2}$, up to a polynomial factor, as $\tau\to-\infty$ in a vertical strip, and so does $E=H-\ee f$, since $f$ has polynomial growth.
\end{proof}

\begin{proof}[Proof of Theorem~\ref{thm:C}~\textup{(ii)}]
By Section~\ref{sec:moments} and the definition of $G$,
\[
\ell=\sqrt{\pi/2}\,g=\sqrt\pi\bigl(A(z)-A(z-1)\bigr),\qquad \sqrt\pi\,G=\sqrt\pi\,\ee^{i\pi z}\bigl(A(z)+A(z-1)\bigr).
\]
Both satisfy \eqref{eq:fe-g} by Theorem~\ref{thm:A} and Lemma~\ref{lem:discarded}, hence so do $S_\pm=\frac12\bigl(\ell\pm\sqrt\pi\,G\bigr)$, and $\ell=S_++S_-$. At an integer $m$, we have $\ee^{i\pi m}=(-1)^m$, so that
\[
(-1)^mS_+(m)=\begin{cases}\sqrt\pi\,A(m)&(m\text{ even}),\\ \sqrt\pi\,A(m-1)&(m\text{ odd}),\end{cases}\qquad
(-1)^mS_-(m)=\begin{cases}-\sqrt\pi\,A(m-1)&(m\text{ even}),\\ -\sqrt\pi\,A(m)&(m\text{ odd}).\end{cases}
\]
The values of $\sqrt\pi\,A$ at the integers follow from $\Gamma(k+\frac12)=\sqrt\pi\,(2k)!/(4^kk!)$. For $k\ge0$ and $k\ge1$, respectively,
\[
\sqrt\pi\,A(2k)=\frac{\sqrt\pi\,k!}{\Gamma(k+\frac12)}=\frac{4^k}{\binom{2k}{k}},\qquad
\sqrt\pi\,A(2k-1)=\frac{\sqrt\pi\,\Gamma(k+\frac12)}{(k-1)!}=\pi\,\frac{k\binom{2k}{k}}{4^k},
\]
and $A(-1)=0$. Comparing with \eqref{eq:pq} gives $(-1)^mS_+(m)=p_m$ and $(-1)^mS_-(m)=-\pi q_m$ for every $m\in\N$, the odd case using the relations $p_{2k+1}=p_{2k}$ and $q_{2k-1}=q_{2k}$ of \eqref{eq:pq}. Adding the two relations gives the second identity in \eqref{eq:arith}. Subtracting them gives $p_m+\pi q_m=(-1)^m\sqrt\pi\,G(m)=\sqrt\pi\bigl(A(m)+A(m-1)\bigr)$. The function $G$ grows like $\ee^{\pi|\tau|}|z|^{1/2}$ as $\tau\to-\infty$, so neither $S_+$ nor $S_-$ has polynomial growth on vertical strips.
\end{proof}

The proof shows that $\{p_m,\pi q_m\}=\{\sqrt\pi\,A(m),\sqrt\pi\,A(m-1)\}$, so the relation $A(m)A(m-1)=m/2$ of Section~\ref{sec:moments} becomes $p_mq_m=m/2$. The alternation of the coefficients, $p_{2k}=p_{2k+1}$ and $q_{2k-1}=q_{2k}$, comes from the fact that $\sqrt\pi\,A$ takes rational values at the even integers and values in $\pi\Q$ at the odd integers.

The linear forms \eqref{eq:arith} give rational approximations of $\ee$ and $\pi$, and Theorem~\ref{thm:A} gives their rate of convergence.

\begin{corollary}\label{cor:approx}
As $m\to\infty$,
\[
\frac{(m+1)!}{D_{m+1}}-\ee\sim(-1)^m\frac{\ee^2}{m\,(m+1)!},\qquad \frac{p_m}{q_m}-\pi\sim(-1)^m\frac{\pi}{2m} .
\]
\end{corollary}

\begin{proof}
By \eqref{eq:arith}, $(m+1)!/D_{m+1}-\ee=(-1)^m\ee f(m)/D_{m+1}$ and $p_m/q_m-\pi=(-1)^m\sqrt{\pi/2}\,g(m)/q_m$. Theorem~\ref{thm:A} gives $f(m)\sim1/m$ and $g(m)\sim1/(2\sqrt m)$. Moreover $D_{m+1}\sim(m+1)!/\ee$, and $q_m\sim\sqrt{m/(2\pi)}$ by $\binom{2k}{k}\sim4^k/\sqrt{\pi k}$.
\end{proof}

The factorial rate for $\ee$ and the rate $1/m$ for $\pi$ are the two behaviors announced in Section~\ref{sec:intro-symmetry}. They correspond to the two Newton polygons compared in Section~\ref{sec:paths}.

\section{The two worlds}\label{sec:paths}

\subsection{Types of solutions}

In the classical theory of linear difference equations with polynomial coefficients, formal solutions may involve fractional powers, logarithms, and further exponential factors \cite{BirkhoffTrjitzinsky1933}. For the two equations considered here, a basis of formal solutions is made of expressions $\Gamma(z)^{\kappa}\lambda^{z}z^{\mu}$ multiplied by series in negative powers of $z$, where the possible values of $\kappa$ are the slopes of the Newton polygon of the operator and $\lambda$ runs over the roots of the characteristic equations attached to its edges. Both equations have explicit bases of solutions over the field of $1$-periodic meromorphic functions, namely $\{f,H\}$ for \eqref{eq:fe-f} and $\{g,G\}$ for \eqref{eq:fe-g}, with nonvanishing Casoratians by Lemma~\ref{lem:discarded}.

In the world of $\ee$, $f(z)\sim1/z$ has $\kappa=0$ and $\lambda=1$, while $H(z)=\ee^{i\pi z}\Gamma(z+2)$ has $\kappa=1$ and $\lambda=-1$. The two values of $\kappa$ are read on the operator $y(z+2)+(z+2)\,y(z+1)-(z+2)\,y(z)$. Balancing its two terms of degree one forces $y(z+1)/y(z)\to1$, and balancing $y(z+2)$ with $(z+2)\,y(z+1)$ forces $y(z+1)/y(z)\sim-z$. In the world of $\pi$, both $g(z)\sim1/(2\sqrt z)$ and $G(z)\sim\ee^{i\pi z}\sqrt{2z}$ have $\kappa=0$, and they differ only by $\lambda=\pm1$, the roots of $\lambda^2-1=0$ obtained by balancing the two terms $(z+1)\,y(z+2)$ and $(z+1)\,y(z)$ of $(z+1)\,y(z+2)+y(z+1)-(z+1)\,y(z)$, the middle term being of lower degree.

This difference explains the rates of Corollary~\ref{cor:approx}. By the proof of Theorem~\ref{thm:C}, the coefficient $D_{m+1}$ is $(-1)^mE(m)/\ee$, and $\pi q_m$ is $\frac12\bigl(\sqrt\pi\,\tilde g(m)-(-1)^m \ell(m)\bigr)$. Up to bounded factors, the error of each approximation is thus the quotient of the connection function by the discarded solution at $m$. This quotient decays factorially when the discarded solution has $\kappa=1$ and like $1/m$ when it has $\kappa=0$. These types agree with the stationary phase formula of Garc\'{\i}a L\'opez \cite[Proposition~2 and Theorem~4]{GarciaLopez2023}, by which the formal type at infinity of the Mellin transform of a holonomic module is determined by the local behavior of the module at its singular points, at $0$, and at infinity. A singular point $s\ne0$ contributes solutions with $\kappa=0$ and $\lambda=s$, and an irregular singularity at infinity contributes solutions with $\kappa\ne0$. In the world of $\pi$, the weight $\varphi_\pi(t)=(1-t)^{1/2}(1+t)^{-1/2}$ satisfies $(1-t^2)\varphi_\pi'+\varphi_\pi=0$, whose singular points $\pm1$ match the values $\lambda=\pm1$. In the world of $\ee$, the weight $\varphi_\ee(t)=(1-t)\,\ee^{t}$ is entire. The value $\lambda=1$ of $f$ comes from the endpoint $t=1$ of the integration, which is a singular point of the module generated by the weight truncated to $[0,1]$, and the irregular singularity of $\ee^{t}$ at infinity matches the value $\kappa=1$ of $H$.

\subsection{Integration paths}

The discarded solutions are integrals of the same functions as the connection functions, taken along other paths.

\begin{proposition}\label{prop:paths}
In the integrals below, $t^{w}$ is continued from $(0,1]$ through the closed upper half-plane, so that $t^{w}=\ee^{i\pi w}|t|^{w}$ for $t<0$. For $\Re z>-2$,
\[
H(z)=\int_0^{-\infty}t^{z+1}\ee^{t}\,\dd t,\qquad E(z)=\int_1^{-\infty}t^{z+1}\ee^{t}\,\dd t,
\]
where the second path avoids $0$. For $\Re z>0$, with $\varphi_\pi(t)=(1-t)^{1/2}(1+t)^{-1/2}$ as in Section~\ref{sec:moments},
\[
S_-(z)=\frac z2\int_{-1}^{1}t^{z-1}\varphi_\pi(t)\,\dd t,\qquad S_+(z)=\frac z2\left(\int_0^{1}-\int_{-1}^{0}\right)t^{z-1}\varphi_\pi(t)\,\dd t .
\]
\end{proposition}

\begin{proof}
On the negative axis, put $t=\ee^{i\pi}u$ with $u>0$. Then
\[
\int_0^{-\infty}t^{z+1}\ee^{t}\,\dd t=\ee^{i\pi(z+2)}\int_0^\infty u^{z+1}\ee^{-u}\,\dd u=H(z).
\] Since $\ee f(z)=I(z+1)=\int_0^1t^{z+1}\ee^{t}\,\dd t$ by Section~\ref{sec:moments}, Cauchy's theorem gives $H-\ee f=\int_1^{-\infty}t^{z+1}\ee^{t}\,\dd t$, which is $E$. In the world of $\pi$, the same substitution gives
\[
z\int_{-1}^{0}t^{z-1}\varphi_\pi(t)\,\dd t=-\ee^{i\pi z}\,z\int_0^1u^{z-1}(1+u)(1-u^2)^{-1/2}\,\dd u=-\ee^{i\pi z}\,z\bigl(J(z-1)+J(z)\bigr),
\]
which equals $-\sqrt\pi\,G(z)$ by the identities $zJ(z-1)=\sqrt\pi A(z)$ and $zJ(z)=\sqrt\pi A(z-1)$ of Section~\ref{sec:moments}. Since $\ell=z\int_0^1t^{z-1}\varphi_\pi(t)\,\dd t$ and $S_\pm=\frac12(\ell\pm\sqrt\pi\,G)$, the two formulas follow.
\end{proof}

The second path in the world of $\ee$ avoids $0$, which gives another proof that $E$ is entire. At an integer $m\ge0$ the function $t^{m+1}$ is single-valued, and an antiderivative of $t^{m+1}\ee^{t}$ is $\ee^{t}Q(t)$ with $Q\in\Z[t]$. Hence $H(m)=-Q(0)\in\Z$ and $E(m)=-\ee\,Q(1)\in\ee\Z$. The integral part of the linear form comes from the path joining $0$ to $-\infty$, and the factor $\ee$ from the endpoint $1$. In the world of $\pi$, for $m\ge1$, $S_-(m)$ is a combination of the integrals $\int_{-1}^{1}t^{k}(1-t^2)^{-1/2}\,\dd t$, which vanish for odd $k$ and lie in $\pi\Q$ for even $k$, while $S_+(m)$ equals $mJ(m-1)$ or $-mJ(m)$ according to the parity of $m$, an integral $J(k)$ with $k$ odd, which is rational. For $m=0$, the limits give $S_+(0)=1$ and $S_-(0)=0$. The factor $\pi$ comes from the path joining the two branch points $\pm1$ of the weight.

\subsection{Borel--Laplace representations}

The substitution $t=\ee^{-u}$ in the integrals of Theorem~\ref{thm:A}, after an integration by parts for $g$, gives
\[
f(z)=\int_0^\infty\ee^{-zu}\exp\bigl(-2u+\ee^{-u}-1\bigr)\,\dd u\quad(\Re z>-2),\qquad g(z)=\int_0^\infty\ee^{-zu}\,\hat g(u)\,\dd u\quad(\Re z>-1),
\]
with $\hat g(u)=\sqrt{2/\pi}\;\ee^{-u}(1-\ee^{-u})^{-1/2}(1+\ee^{-u})^{-3/2}$, the branches being positive for $u>0$. The kernel of $f$ is entire, and it is the Borel transform of the asymptotic expansion of $f$ in powers of $1/z$. The kernel $\hat g$ behaves like $(2\sqrt\pi)^{-1}u^{-1/2}$ at $0$, it is the Borel transform of the expansion of $g$ in half-integer powers of $1/z$ with the convention $z^{-\beta}\mapsto u^{\beta-1}/\Gamma(\beta)$, and its singularities lie in $\pi i\Z$.

\section{The holonomic triangle}\label{sec:branch}

\subsection{Proof of Proposition~\texorpdfstring{\ref{prop:D}}{1.4}}

\begin{proof}[Proof of Proposition~\ref{prop:D}]
Since $\varphi$ is holomorphic near $0$, there are $r_0>0$ and coefficients $\varphi_k$ with $\varphi(t)=\sum_{k\ge0}\varphi_kt^k$ for $|t|<r_0$. Fix $0<\delta<\min(r_0,1)$. For $\Re z>0$,
\begin{equation}\label{eq:split}
h(z)=z\sum_{k\ge0}\varphi_k\,\frac{\delta^{z+k}}{z+k}+z\int_\delta^1t^{z-1}\varphi(t)\,\dd t .
\end{equation}
The integral is an entire function of $z$, because $t\ge\delta$ on the range of integration and $\varphi$ is integrable at $1$ when $\Re\alpha>-1$. The series converges locally uniformly on $\C\setminus\{0,-1,-2,\ldots\}$, since $|\varphi_k|\delta^k$ decays geometrically. Hence $h$ extends meromorphically to $\C$ with at most simple poles at the negative integers, and the factor $z$ removes the pole at $z=0$, where the value is $\varphi_0=\varphi(0)$.

In a strip $a\le\Re z\le b$ with $a>-1$, we have $|z+k|\ge a+k>0$ for $k\ge1$, so that $|z|/|z+k|\le1+k/(a+k)$ is bounded by a constant depending only on $a$. The series in \eqref{eq:split} is therefore bounded in this strip, and the last term is at most $|z|\int_\delta^1t^{a-1}|\varphi(t)|\,\dd t$. This gives polynomial growth on vertical strips in $\Re z>-1$.

Assume now that $L\varphi=0$ with $L=\sum_{i=0}^{s}c_i(t)\,\partial_t^{\,i}$, where the $c_i$ are polynomials and $c_s\ne0$, and put $M(w)=\int_0^1t^{w}\varphi(t)\,\dd t$, so that $h(z)=zM(z-1)$. For $\Re z>s$, we multiply $L\varphi=0$ by $t^{z}(1-t)^{s}$ and integrate over $[0,1]$. Moving the derivatives onto $t^z(1-t)^sc_i(t)$ by repeated integration by parts produces boundary terms of the form $\partial_t^{\,p}\bigl(t^z(1-t)^sc_i(t)\bigr)\,\varphi^{(q)}(t)$ with $p+q=i-1$. At $t=0$ they vanish because $\Re z>s$. At $t=1$ they are $O\bigl((1-t)^{s-p+\Re\alpha-q}\bigr)=O\bigl((1-t)^{1+\Re\alpha}\bigr)$, since $\varphi^{(q)}(t)=O\bigl((1-t)^{\Re\alpha-q}\bigr)$, and they vanish as well. We obtain
\[
\int_0^1\varphi(t)\,L^{*}\bigl(t^z(1-t)^s\bigr)\,\dd t=0,\qquad L^{*}y=\sum_{i=0}^{s}(-1)^i\partial_t^{\,i}\bigl(c_iy\bigr).
\]
By the Leibniz rule $L^{*}\bigl(t^z(1-t)^s\bigr)=t^{z-s}\sum_{j=0}^{J}P_j(z)\,t^{j}$ for some $J$ and polynomials $P_j$, so that $\sum_{j}P_j(z)M(z-s+j)=0$. The term of highest degree in $z$ comes from $i=s$ and equals $(-1)^sz^s\,c_s(t)(1-t)^s\,t^{z-s}$, which is not zero, so the $P_j$ are not all zero. Substituting $M(w)=h(w+1)/(w+1)$ and clearing denominators gives a nonzero linear relation between shifts of $h$ with polynomial coefficients. It holds for $\Re z>s$, hence on $\C$ by meromorphic continuation.
\end{proof}

The two weights of Theorem~\ref{thm:A} correspond to $\varphi_{\ee}(t)=(1-t)\,\ee^{t}$, with $\alpha=1$, and $\varphi_{\pi}(t)=(1-t)^{1/2}(1+t)^{-1/2}$, with $\alpha=1/2$. Both are D-finite, and the equations \eqref{eq:fe-f} and \eqref{eq:fe-g} are computed directly in Section~\ref{sec:moments}.

\subsection{Proof of Proposition~\texorpdfstring{\ref{prop:E}}{1.5}}

\begin{proof}[Proof of Proposition~\ref{prop:E}]
Let $K_0$ be the set of indices $k$ with $S_k\ne0$. As a rational function of $n$ with coefficients in $\C(z)$, each $S_k$ with $k\in K_0$ can be written $S_k(n,z)=c_k(z)\,n^{d_k}\bigl(1+O(1/n)\bigr)$ with $d_k\in\Z$ and $c_k\ne0$, uniformly for $z$ in compact sets avoiding finitely many points. For the factor $\Lambda$,
\[
\frac{\Lambda(n,z+k)}{\Lambda(n,z)}=n^{ak}\prod_j\left(\frac{\Gamma(n+\alpha_jz+\beta_j+\alpha_jk)}{\Gamma(n+\alpha_jz+\beta_j)}\right)^{m_j}=n^{\xi k}\bigl(1+O(1/n)\bigr),
\]
uniformly on compact sets, by the expansion $\Gamma(w+c)/\Gamma(w)=w^{c}\bigl(1+O(1/w)\bigr)$ \cite[5.11.12]{DLMF} applied with $w=n+\alpha_jz+\beta_j$. The factor $\lambda^n$ cancels because $\lambda$ does not depend on $z$. Dividing the relation by $\Lambda(n,z)$ and inserting the asymptotics of $u_n(z+k)$ gives, on a nonempty open set $\mathcal O$,
\[
\sum_{k\in K_0}c_k(z)\,n^{d_k+\xi k}\bigl(h(z+k)+O(1/n)\bigr)=0 .
\]
Let $D^{*}$ be the largest of the real numbers $d_k+\xi k$ and $K^{*}$ the set of indices where it is attained. Dividing by $n^{D^{*}}$ and letting $n\to\infty$ yields $\sum_{k\in K^{*}}c_k(z)h(z+k)=0$ on $\mathcal O$, hence on $\C$. If $k\ne l$ both belong to $K^{*}$, then $\xi(k-l)=d_l-d_k\in\Z$. When $\xi\notin\Q$ this is impossible, so $K^{*}$ is a single index $k^{*}$, and $c_{k^{*}}(z)h(z+k^{*})=0$ forces $h\equiv0$. When $\xi=p/q$ in lowest terms, $q$ divides $k-l$ for all $k,l\in K^{*}$. If moreover $h\not\equiv0$, the set $K^{*}$ has at least two elements by the same argument, and $\sum_{k\in K^{*}}c_k(z)\,h(z+k)=0$ is the announced nonzero equation.
\end{proof}

For the family $(z/2)_n/n!\sim n^{z/2-1}/\Gamma(z/2)$, the relation $\tfrac z2\,u_n(z+2)=\bigl(n+\tfrac z2\bigr)u_n(z)$ has exponents $0+2\cdot\tfrac12=1$ for $k=2$ and $1$ for $k=0$, and the proof of Proposition~\ref{prop:E} gives $h(z+2)=(2/z)\,h(z)$, which $h=1/\Gamma(z/2)$ satisfies.

\subsection{Proof of Theorem~\texorpdfstring{\ref{thm:T}}{1.6}}

\begin{proof}[Proof of Theorem~\ref{thm:T}]
The coefficients depend on $n$ and $z$ only through $n+z$, so the sequence $n\mapsto u_{n+1}(z;v)$ satisfies the recurrence with the parameter $z+1$, and its initial vector is $T(z)v$. Hence $u_n\bigl(z+1;T(z)v\bigr)=u_{n+1}(z;v)$ for $n\ge1$. We divide by $\Lambda(n,z+1)$, write it as $\Lambda(n+1,z)$ times $\Lambda(n,z+1)/\Lambda(n+1,z)$, and let $n\to\infty$. This gives $\mathbf h(z+1)T(z)v=\eta(z)\mathbf h(z)v$ for every $v$. The matrix $T(z)$ is a companion matrix with entries in $\C(z)$, and its determinant is $\pm p_0(z+1)/p_r(z+1)$, which is not zero. Iterating, we obtain $\mathbf h(z+k)=\mathbf h(z)M_k(z)$ with $M_k\in\mathrm{GL}_r\bigl(\C(z)\bigr)$. For $v\in\C(z)^r$, the functions $h(z+k)=\mathbf h(z)M_k(z)v(z+k)$ with $0\le k\le r$ therefore lie in the vector space over $\C(z)$ spanned by the $r$ components of $\mathbf h(z)$, and a nontrivial relation between them is a functional equation of order at most $r$. The relation $\mathbf h(z)=\eta(z)^{-1}\mathbf h(z+1)T(z)$ continues $\mathbf h$, hence $h$, meromorphically to the left.
\end{proof}

For the worlds of $\ee$ and $\pi$, with $v=(0,1)$, the relation $\mathbf h(z+1)T(z)=\mathbf h(z)$ gives back \eqref{eq:fe-f} and \eqref{eq:fe-g}. For the second recurrence, for instance, $T(z)$ sends $(1,0)$ to $(0,1)$ and $(0,1)$ to $(1,\frac1{z+1})$, so that the two components $h_1,h_2$ of $\mathbf h$ satisfy $h_1(z)=h_2(z+1)$ and $h_2(z)=h_1(z+1)+h_2(z+1)/(z+1)$, which is \eqref{eq:fe-g} for $g=h_2$.

\begin{remark}\label{rem:affine}
Some structure is needed in Theorem~\ref{thm:T}. For $\alpha>0$, the recurrence $(n+\alpha)\,u_{n+2}=(n+\alpha-z)\,u_{n+1}+z\,u_n$, with $u_1=0$ and $u_2=1$, has coefficients which are affine in $n$ and $z$ separately. Its solution is $u_n(z)=\sum_{j=0}^{n-2}(-z)^j/(\alpha+1)_j$, which tends to $h_\alpha(z)={}_1F_1(1;\alpha+1;-z)=\alpha\int_0^1\ee^{-zt}(1-t)^{\alpha-1}\,\dd t$. If $\alpha$ is not an integer, $h_\alpha$ satisfies no nonzero linear functional equation in integer shifts with rational coefficients. To see this, let $w=\alpha(1-t)^{\alpha-1}\mathbf 1_{[0,1]}$. A relation $\sum_kR_k(z)\,h_\alpha(z+k)=0$ with polynomials $R_k$ gives $\sum_kR_k(D_t)\bigl(\ee^{-kt}w\bigr)=0$ by the injectivity of the Laplace transform on distributions with compact support. On $(0,1)$ this reads $\sum_k\ee^{-kt}r_k(t)=0$ with $r_k=R_k(D_t-k)w/w\in\C(t)$. The functions $\ee^{-kt}$ are linearly independent over $\C(t)$, so every $r_k$ vanishes. This is impossible, since $D_t^mw/w$ has a pole of order exactly $m$ at $t=1$ when $\alpha\notin\Z$.
\end{remark}

\subsection{Proof of Theorem~\texorpdfstring{\ref{thm:F}}{1.7}}

\begin{proof}[Proof of Theorem~\ref{thm:F}]
Let $\theta=x\,\dd/\dd x$ and $\mathcal U_k(x)=\sum_{j=1}^{k}u_jx^j$, with $\mathcal U_0=0$. Multiplying \eqref{eq:classF} by $x^{n+r}$ and summing over $n\ge1$, the term of index $k$ gives
\[
\sum_{j\ge k+1}\bigl(\alpha_k(j-k+z)+\beta_k\bigr)u_j\,x^{j+r-k}=x^{r-k}\bigl(\alpha_k(\theta+z-k)+\beta_k\bigr)\bigl(U_z-\mathcal U_k\bigr).
\]
Summing over $k$ gives $P(x)(\theta+z)U_z+Q(x)U_z=R(x,z)$ with $R(x,z)=\sum_{k=0}^{r}x^{r-k}\bigl(\alpha_k(\theta+z-k)+\beta_k\bigr)\mathcal U_k$, which is (i). Each $\mathcal U_k$ is divisible by $x$, hence so is $R$.

The function $(Q-aP)/(xP)$ is holomorphic at $0$ because $Q(0)=aP(0)$, so $\nu$ is holomorphic near $0$ with $\nu(0)=1$ and $\nu'/\nu=(Q-aP)/(xP)$. The function $\mu(x)=x^{z+a}\nu(x)$ then satisfies $xP\mu'=(zP+Q)\mu$, and (i) becomes $(\mu U_z)'=\mu R/(xP)$. Near $0$ we have $\mu U_z=O\bigl(x^{\Re z+\Re a+1}\bigr)$, which tends to $0$ when $\Re z>-1-\Re a$. Therefore, for $|x|<\rho$,
\begin{equation}\label{eq:solF}
U_z(x)=\frac{x^{-z-a}}{\nu(x)}\int_0^xt^{z+a-1}\,\frac{\nu(t)\,R(t,z)}{P(t)}\,\dd t ,
\end{equation}
where the integrand is $O\bigl(t^{\Re z+\Re a}\bigr)$ at $0$.

The singularities of $U_z$ are among the roots of $P$. Near a simple root $\zeta$, the function $(Q-aP)/(tP)$ has a simple pole with residue $\gamma_\zeta=Q(\zeta)/(\zeta P'(\zeta))$, so that $\nu(x)=(\zeta-x)^{\gamma_\zeta}\nu_\zeta(x)$ with $\nu_\zeta$ holomorphic and nonzero at $\zeta$. At $\zeta=\rho$, the integrand in \eqref{eq:solF} is $O\bigl((\rho-t)^{\Re\gamma-1}\bigr)$, which is integrable since $\Re\gamma>0$. We write $\int_0^x=\int_0^\rho-\int_x^\rho$. The second part is $(\rho-x)^{\gamma}$ times a function holomorphic at $\rho$, and after division by $\nu(x)$ it contributes a function holomorphic at $\rho$. The first part gives
\[
U_z(x)=\frac{\rho^{-z-a}}{K}\,(\rho-x)^{-\gamma}\bigl(1+O(\rho-x)\bigr)\int_0^\rho t^{z+a-1}\frac{\nu(t)R(t,z)}{P(t)}\,\dd t+(\text{a function holomorphic at }\rho).
\]
Here $K=\lim\nu(x)(\rho-x)^{-\gamma}=\nu_\rho(\rho)$ is not zero. At a root $\zeta\ne\rho$ with $|\zeta|=\rho$, the same decomposition shows that $U_z$ is $(\zeta-x)^{-\gamma_\zeta}$ times a function holomorphic at $\zeta$, plus a function holomorphic at $\zeta$, with a logarithmic term when $\gamma_\zeta$ is an integer at most $0$. Its contribution to $u_n(z)$ is $O\bigl(\rho^{-n}n^{\Re\gamma_\zeta-1}\log n\bigr)$, which is $o\bigl(\rho^{-n}n^{\Re\gamma-1}\bigr)$ by the assumption $\Re\gamma_\zeta<\Re\gamma$. The roots with $|\zeta|>\rho$ contribute exponentially smaller terms. Singularity analysis \cite[Theorems~VI.1 and~VI.3 and Section~VI.5]{FlajoletSedgewick2009}, applied to $(\rho-x)^{-\gamma}=\rho^{-\gamma}(1-x/\rho)^{-\gamma}$, gives
\[
\rho^{n}n^{1-\gamma}u_n(z)\longrightarrow\frac{\rho^{-z-a-\gamma}}{K\,\Gamma(\gamma)}\int_0^\rho t^{z+a-1}\frac{\nu(t)R(t,z)}{P(t)}\,\dd t ,
\]
and the change of variables $t=\rho s$ gives (ii), the factor $\rho^{z+a}$ cancelling. The singular points and the exponents do not depend on $z$, and the coefficients of the local expansions are holomorphic in $z$, so the transfer estimates are uniform for $z$ in a compact subset of the half-plane, and the convergence is locally uniform.

For (iii), write $R=R_0+zR_1$ and $\psi_j(s)=\nu(\rho s)R_j(\rho s)/P(\rho s)$ for $j=0,1$. Then
\[
h(z)=\frac{1}{K\,\Gamma(\gamma)\,\rho^{\gamma}}\bigl(M_0(z+a)+z\,M_1(z+a)\bigr),\qquad M_j(w)=\int_0^1s^{w-1}\psi_j(s)\,\dd s .
\]
The functions $\psi_j$ are hyperexponential, hence D-finite. They are holomorphic on a neighborhood of $[0,1)$, since the roots of $P$ other than $\rho$ have modulus at least $\rho$ and are not in $(0,\rho)$, and they are $(1-s)^{\gamma-1}$ times a function holomorphic at $1$. Proposition~\ref{prop:D}, with $\alpha=\gamma-1$, applies to $wM_j(w)$. Hence $M_0(z+a)$ and $zM_1(z+a)$ extend meromorphically to $\C$, and the shifts of each of them span a vector space of finite dimension over $\C(z)$. The shifts of $h$ lie in the sum of these two spaces, which gives a nonzero linear relation between shifts of $h$ with polynomial coefficients.
\end{proof}

\begin{example}\label{ex:F}
The recurrence
\[
\bigl(n+z+\tfrac12\bigr)u_{n+2}-\bigl(3(n+z)+2\bigr)u_{n+1}+\bigl(2(n+z)+\tfrac25\bigr)u_n=0,\qquad u_1=\tfrac3{10},\quad u_2=1,
\]
belongs to the class \eqref{eq:classF} with $P(x)=(1-x)(1-2x)$, $\rho=\tfrac12$, $\gamma=\tfrac95$, $a=-\tfrac32$, $\nu(x)=(1-2x)^{9/5}(1-x)^{-1/10}$, and $K=2^{19/10}$. Put $C_*=2^{-1/10}/\Gamma(\frac95)$ and $q(t)=(1-t)^{4/5}(1-t/2)^{-11/10}$. Theorem~\ref{thm:F} gives $2^{-n}n^{-4/5}u_n(z)\to h(z)$ with
\[
h(z)=\frac{C_*}{40}\int_0^1t^{z-3/2}\,q(t)\,\bigl(6z-3+(z-1)\,t\bigr)\,\dd t\qquad(\Re z>\tfrac12).
\]
Let $P_2(z)=(5z+11)(35z+27)$, $P_1(z)=-5(105z^2+361z+301)$, and $P_0(z)=5(2z+3)(35z+62)$. The equation
\[
P_2(z)\,h(z+2)+P_1(z)\,h(z+1)+P_0(z)\,h(z)=0
\]
was first found by linear algebra on numerical values of $h$. It follows from the identity
\begin{multline*}
\sum_{k=0}^{2}P_k(z)\,t^{z+k-3/2}q(t)\bigl(6(z+k)-3+(z+k-1)\,t\bigr)\\
=\frac{\dd}{\dd t}\Bigl(t^{z-1/2}q(t)\,(1-t)(2-t)\bigl(3P_0(z)+5(35z+27)(z+1)\,t\bigr)\Bigr),
\end{multline*}
which one checks by expanding the derivative with $q'/q=-(5+3t)/\bigl(10(1-t)(2-t)\bigr)$, and from an integration over $[0,1]$, the boundary terms vanishing at both ends for $\Re z>\frac12$. An integration by parts gives a second representation,
\[
h(z)=\frac{C_*}{200}\int_0^1t^{z+1/2}\,(17-3t)\,(1-t)^{-1/5}(1-t/2)^{-21/10}\,\dd t\qquad(\Re z>-\tfrac32).
\]
With $(17-3t)(1-t/2)^{-21/10}=\sum_{j\ge0}d_jt^j$, whose radius of convergence is $2$, this gives $h(z)=\frac{C_*}{200}\sum_{j\ge0}d_j\,B\bigl(z+j+\frac32,\frac45\bigr)$ with $\sum_j|d_j|<\infty$. Since $|B(w,\frac45)|\le C|w|^{-4/5}$ uniformly in $\Re w\ge1$, we obtain $h(z)=O\bigl(|z|^{-4/5}\bigr)$ uniformly in $\Re z\ge1$, and $h(z)\sim\frac{7}{20}z^{-4/5}$ as $z\to+\infty$. The weight of Theorem~\ref{thm:G} is
\[
\varphi(t)=\frac{C_*}{200}\int_t^1s^{1/2}(17-3s)(1-s)^{-1/5}(1-s/2)^{-21/10}\,\dd s=\varphi(0)+t^{3/2}A_{3/2}(t),\qquad A_{3/2}(0)=-\frac{17\,C_*}{300},
\]
with $A_{3/2}$ holomorphic in $|t|<1$. The Taylor coefficients of $(17-3s)(1-s)^{-1/5}(1-s/2)^{-21/10}$ are $O(j^{-4/5})$, so the coefficients of $A_{3/2}$ are $O(j^{-9/5})$. The hypotheses of Proposition~\ref{prop:partial} hold with $\mathcal E=\{0,\frac32\}$ and $A_0=\varphi(0)$, and the partial sums $s_n(z)$ of this proposition converge to the same function $h$ as $2^{-n}n^{-4/5}u_n(z)$, although the two families are different.
\end{example}

\subsection{Proof of Theorem~\texorpdfstring{\ref{thm:G}}{1.8}}

\begin{proof}[Proof of Theorem~\ref{thm:G}]
(i) For $\Re z\ge c'$, the function $h(z)/z$ is $O\bigl(|z|^{-1-\delta}\bigr)$. The integral defining $\varphi$ therefore converges absolutely, $\varphi$ is continuous on $(0,\infty)$, and $|\varphi(t)|\le C_1t^{-c'}$. For $t>1$ and $X>c'$, the integrand is holomorphic in the strip $c'\le\Re z\le X$ and is $O\bigl(t^{-\Re z}|z|^{-1-\delta}\bigr)$ there, so the line of integration can be moved to $\Re z=X$. This gives $|\varphi(t)|\le C_1t^{-X}$ for every $X>c'$, hence $\varphi(t)=0$. The same argument shows that the integral does not depend on the choice of the line in $\Re z\ge c'$. For $c_1>c'$, the function $t^{c_1-1}\varphi(t)$ is integrable on $(0,1)$, and the Fourier inversion theorem in the variable $u=-\log t$ gives $h(z)=z\int_0^1t^{z-1}\varphi(t)\,\dd t$ on the line $\Re z=c_1$. Both sides are holomorphic in $\Re z>c'$, so the identity holds there. Uniqueness follows from the injectivity of the Mellin transform.

(ii) Put $\tilde h(z)=h(z)/z$, and let $\mathcal T$ be the distribution on $(0,\infty)$ given by $\varphi$, extended by $0$ for $t>1$. For $\chi\in C_c^\infty(0,\infty)$, the function $\Phi_\chi(z)=\int_0^\infty t^{-z}\chi(t)\,\dd t$ is entire, and repeated integrations by parts show that it decays faster than any power of $|\tau|$ on every vertical line. By (i) and Fubini's theorem, for $\sigma>c'$ and $k\ge0$,
\[
\langle t^k\mathcal T,\chi\rangle=\frac{1}{2\pi i}\int_{\Re z=\sigma}\tilde h(z+k)\,\Phi_\chi(z)\,\dd z .
\]
Since $\Phi_{\theta\chi+\chi}=z\,\Phi_\chi$ and $\langle-\theta S,\chi\rangle=\langle S,\theta\chi+\chi\rangle$ for every distribution $S$, we obtain
\[
\langle L\mathcal T,\chi\rangle=\frac{1}{2\pi i}\int_{\Re z=\sigma}\sum_{k,j}c_{kj}\,z^j\,\tilde h(z+k)\,\Phi_\chi(z)\,\dd z=\frac{1}{2\pi i}\int_{\Re z=\sigma}\sum_kR_k(z)\,h(z+k)\,\Phi_\chi(z)\,\dd z=0 .
\]
Hence $L\mathcal T=0$, and on $(0,1)$ this says $L\varphi=0$ in the sense of distributions. On an open interval where the leading coefficient of $L$ does not vanish, a distribution solution is a classical solution, hence real-analytic. If $J$ is the largest index $j$ with some $c_{kj}\ne0$, the coefficient of $\partial_t^{\,J}$ in $L$ is $(-1)^Jt^J\sum_kc_{kJ}t^k$, which is not zero, so $L\ne0$. The correspondence between the two operators is the isomorphism $z\mapsto-\theta$, $S\mapsto t$ of the algebraic Mellin transform \cite{LoeserSabbah1991}, and at the integers the same computation is the method of Ablinger \cite{Ablinger2016}, where the terms at $t=1$ appear explicitly. The example $\varphi(t)=(b-t)_+$ with $0<b<1$, for which $h(z)=b^{z+1}/(z+1)$ and $L=(t-b)\,t^2\partial_t^{\,2}$, shows that $\varphi$ need not be analytic at an interior zero of the leading coefficient.

\end{proof}

\subsection{Partial sums}

When the weight of Theorem~\ref{thm:G} has a convergent expansion at $0$, its Taylor coefficients give a second explicit family, made of partial sums of simple fractions.

\begin{proposition}\label{prop:partial}
In the setting of Theorem~\ref{thm:G}, assume that $\varphi(t)=\sum_{\lambda\in\mathcal E}t^{\lambda}A_\lambda(t)$ on $(0,1)$, where $t^{\lambda}=\ee^{\lambda\log t}$, the set $\mathcal E$ is finite and its elements are pairwise incongruent modulo $\Z$, and each $A_\lambda(t)=\sum_{k\ge0}A_{\lambda,k}t^k$ is holomorphic in $|t|<1$ with $A_\lambda(0)\ne0$ and $\sum_k|A_{\lambda,k}|/(k+1)<\infty$. Then
\[
s_n(z)=\sum_{\lambda\in\mathcal E}\sum_{k=0}^{n}\frac{z\,A_{\lambda,k}}{z+\lambda+k}
\]
is P-recursive in $n$ with coefficients in $\C(z)$, its generating function
\[
\sum_{n\ge0}s_n(z)\,x^n=\frac{z}{1-x}\sum_{\lambda\in\mathcal E}x^{-z-\lambda}\int_0^xt^{z+\lambda-1}A_\lambda(t)\,\dd t
\]
is D-finite in $x$, and $s_n(z)\to h(z)$ for $\Re z>c'$.
\end{proposition}

\begin{proof}
The expansion makes $\varphi$ analytic near $0$ in the slit disk, so $L\varphi=0$ there by Theorem~\ref{thm:G}(ii) and analytic continuation. Let $M$ be the continuation along a loop around $0$. The coefficients of $L$ are polynomials, so each $M^j\varphi$ is a solution, and $M^j\varphi=\sum_{\lambda\in\mathcal E}\ee^{2\pi ij\lambda}t^\lambda A_\lambda(t)$. The numbers $\ee^{2\pi i\lambda}$ are distinct, so the Vandermonde matrix $\bigl(\ee^{2\pi ij\lambda}\bigr)$, with $0\le j<|\mathcal E|$ and $\lambda\in\mathcal E$, is invertible, and each $t^\lambda A_\lambda$ is a linear combination of the $M^j\varphi$, hence a solution of $L$. Conjugating $L$ by $t^\lambda$ shows that $A_\lambda$ is D-finite, so $(A_{\lambda,k})_k$ is P-recursive \cite[Chapter~7]{KauersPaule2011}. The factor $z/(z+\lambda+k)$ is a hypergeometric term in $k$ with coefficients in $\C(z)$, so the terms $\sum_\lambda zA_{\lambda,k}/(z+\lambda+k)$ and their partial sums $s_n(z)$ are P-recursive. Integrating term by term, $z\,x^{-z-\lambda}\int_0^xt^{z+\lambda-1}A_\lambda(t)\,\dd t=\sum_kzA_{\lambda,k}x^k/(z+\lambda+k)$, which gives the generating function, and this function is D-finite because $A_\lambda$ is. For $\Re z$ large, the same term by term integration on $[0,1]$, justified by $\sum_k|A_{\lambda,k}|/(k+1)<\infty$, gives $h(z)=\lim_ns_n(z)$. The series converges normally on compact sets which avoid the points $-\lambda-k$ with $A_{\lambda,k}\ne0$, so it is meromorphic there and coincides with $h$. Points coming from different exponents are distinct, since the exponents are incongruent modulo $\Z$, so the poles do not cancel, and none of them lies in $\Re z>c'$, where $h$ is holomorphic. Hence $s_n(z)\to h(z)$ for $\Re z>c'$.
\end{proof}

These partial sums are the analogue, with a parameter and at the singular point $1$ of the operator, of the construction of Huang and Kauers for the values of D-finite functions inside the disk of convergence \cite[Lemma~16]{HuangKauers2018}. The normalization $A_\lambda(0)\ne0$ fixes the family. Without it, writing $t^\lambda A_\lambda=t^{\lambda-1}(tA_\lambda)$ keeps the weight and the limit but changes the truncations.

\subsection{Proof of Theorem~\texorpdfstring{\ref{thm:N}}{1.9} and of Corollary~\texorpdfstring{\ref{cor:closure}}{1.10}}

\begin{proof}[Proof of Theorem~\ref{thm:N}]
Choose an integer $q>M+1$. The function $h(z)/z^{q}$ is $O\bigl(|z|^{-1-\delta}\bigr)$ for $\Re z\ge\sigma$, with $\delta=q-M-1>0$, so the argument of Theorem~\ref{thm:G}(i) applied to $h(z)/z^{q-1}$ gives a continuous function $\varphi$ on $(0,\infty)$, zero for $t>1$, with $|\varphi(t)|\le C_0t^{-\sigma}$ and
\[
h(z)=z^{q}\int_0^1t^{z-1}\varphi(t)\,\dd t\qquad(\Re z>\sigma).
\]
With $\theta=t\,\dd/\dd t$ we have $(z_0+\theta)^qt^j=(z_0+j)^qt^j$, hence
\[
\Delta_k=\int_0^1t^{z_0-1}\varphi(t)\,(z_0+\theta)^q(1-t)^k\,\dd t .
\]
We write $(z_0+\theta)^q=\sum_{r=0}^{q}d_r\,t^r\partial_t^{\,r}$. For $k\ge q$, the term of index $r$ is bounded by $C_0|d_r|\,k^r\,B(z_0-\sigma+r,k-r+1)=O(k^{\sigma-z_0})$, so $|\Delta_k|\le C_1(k+1)^{\sigma-z_0}$. On a compact set $K$, we have $\bigl|\binom{z-z_0}{k}\bigr|=\bigl|(z_0-z)_k/k!\bigr|\le C_K(k+1)^{z_0-\Re z-1}$, so the terms of the series are $O\bigl((k+1)^{\sigma-s_K-1}\bigr)$ with $s_K=\min_K\Re z$, and the series converges normally on $K$ when $s_K>\sigma$. For $0<t<1$ we have $\sum_k(-1)^k\binom{z-z_0}{k}(1-t)^k=t^{z-z_0}$, and the same bounds allow the summation under the integral, so that the limit is $\int_0^1t^{z_0-1}\varphi(t)(z_0+\theta)^qt^{z-z_0}\,\dd t=z^q\int_0^1t^{z-1}\varphi(t)\,\dd t=h(z)$. The interpolation property follows from the binomial inversion of $\Delta_k$.

If $\sum_kR_k(z)h(z+k)=0$, the sequence $h(z_0+j)$ satisfies $\sum_kR_k(z_0+j)\,h(z_0+j+k)=0$, so it is P-recursive. With $Y(x)=\sum_jh(z_0+j)x^j$, we have $\sum_k\Delta_kx^k=(1-x)^{-1}Y\bigl(-x/(1-x)\bigr)$, which is D-finite, so $(\Delta_k)$ is P-recursive. The factor $(z_0-z)_k/k!$ is a hypergeometric term in $k$ with coefficients in $\C(z)$, so the terms of $p_n$ and their partial sums are P-recursive, and the generating function $(1-x)^{-1}\sum_k(z_0-z)_k\Delta_kx^k/k!$ is D-finite \cite{Stanley1980}.
\end{proof}

\begin{proof}[Proof of Corollary~\ref{cor:closure}]
In the representation of Theorem~\ref{thm:F}(ii), the functions $R_0$ and $R_1$ vanish at $0$ and $\Re\gamma>0$, so the integrals of $s^{\sigma+\Re a-1}|\nu(\rho s)R_j(\rho s)/P(\rho s)|$ over $(0,1)$ are finite for $\sigma>\max(0,-1-\Re a)$. Since $|s^{z}|=s^{\Re z}\le s^{\sigma}$ for $0<s<1$ and $\Re z\ge\sigma$, we obtain $|h(z)|\le C_\sigma(1+|z|)$ in this half-plane. Theorem~\ref{thm:F}(iii) gives the functional equation, and Theorem~\ref{thm:N} applies.
\end{proof}

\subsection{Euler's Gamma function}

The sequence $w_n(z)=n!/(z)_n$ of the introduction satisfies $(n+z)\,w_n(z+1)=z\,w_n(z)$ and $w_n(z)\sim\Gamma(z)\,n^{1-z}$. Here $\Lambda(n,z)=n^{1-z}$, so $\xi=-1$. The coefficients $S_1=n+z$ and $S_0=-z$ give the exponents $1-1=0$ and $0$, both maximal, and Proposition~\ref{prop:E} gives $h(z+1)=z\,h(z)$, which is $\Gamma(z+1)=z\Gamma(z)$. Together with the inhomogeneous differential equation of order one for the generating function in Table~\ref{tab:examples}, Euler's Gamma function realizes the three vertices of Figure~\ref{fig:triangle}, with a recurrence and a functional equation of order one. The corresponding homogeneous differential equation has order two, since for $z=2$ the generating function $W(x)=-x^{-1}\log(1-x)-1$ satisfies no homogeneous equation of order one. It is not in the class \eqref{eq:classF}, since $z$ does not enter the coefficient $n+1$ through $n+z$, and its factor $\Lambda(n,z)=n^{1-z}$ depends on $z$.

\subsection{Parameters in the coefficients}

When the parameters enter the recurrence as coefficients, the connection constant can be computed in closed form.

\begin{proposition}\label{prop:EG}
Let $1\le s\le r$ be integers, let $z_1,\ldots,z_r\in\C$, $z=z_1+\cdots+z_r$, and $Z(x)=\sum_{i=1}^{r}z_ix^i$. Define $(T_n)_{n\ge1}$ by $T_1=\cdots=T_{r-1}=0$, $T_r=1$, and
\[
T_{n+r}=T_{n+r-s}+\frac{s}{n}\sum_{i=1}^{r}z_i\,T_{n+r-i}\qquad(n\ge1).
\]
Assume that $z\notin\{-1,-2,\ldots\}$ and that $\Re Z(\omega)<\Re z$ for every $s$-th root of unity $\omega\ne1$. Then
\[
\lim_{n\to\infty}\frac{n^{z}}{s\,T_n}=\Gamma(z+1)\prod_{i=1}^{r}c(s,i)^{z_i},\qquad c(s,i)=s\,\exp\bigl(\psi(i/s)+\gamma_E\bigr),
\]
where $\psi=\Gamma'/\Gamma$ is the digamma function and $\gamma_E$ is Euler's constant.
\end{proposition}

For $r=s=1$ this is the Euler-Gauss formula $n^z/T_n\to\Gamma(z+1)$, with $T_n=(z+1)_{n-1}/(n-1)!$. For $(r,s)=(2,1)$ and $(z_1,z_2)=(0,1)$, and for $(r,s)=(2,2)$ and $(z_1,z_2)=(1/2,0)$, the recurrence is one of the two recurrences \eqref{eq:mirror}, and the formula gives back $n/u_n\to\ee$ and $2n/v_n^2\to\pi$, since $c(1,2)=\ee$ and $c(2,1)=1/2$. The proposition was stated without proof in the note of 2004 \cite[Part~D]{Cloitre2004}, where the closed form of the constants $c(s,i)$ is credited to an experimental observation of P.~D.~Hanna.

\begin{proof}
Let $T(x)=\sum_{n\ge1}T_nx^n$ and $\theta=x\,\dd/\dd x$. Multiplying the recurrence by $n\,x^{n+r}$, summing over $n\ge1$, and using $T_k=0$ for $k<r$, we obtain $(\theta-r)T=x^s(\theta-r+s)T+s\,Z(x)\,T$, that is
\[
\frac{T'(x)}{T(x)}=\frac rx+\frac{s\,x^{s-1}}{1-x^s}+\frac{s\,Z(x)}{x\,(1-x^s)} .
\]
Since $Z(x)/x$ is a polynomial and $T(x)\sim x^r$ at $0$, this gives
\[
T(x)=\frac{x^r}{1-x^s}\,\exp\left(\int_0^x\frac{s\,Z(t)}{t\,(1-t^s)}\,\dd t\right)\qquad(|x|<1).
\]
The rational function $sZ(t)/\bigl(t(1-t^s)\bigr)$ is the sum of a polynomial and of the simple elements $Z(\omega)/(\omega-t)$, one for each $s$-th root of unity $\omega$, because $1-t^s=s(1-t/\omega)+O\bigl((t-\omega)^2\bigr)$ near $\omega$. Hence $T(x)=x^r\,\ee^{Q(x)}\prod_{\omega^s=1}(1-x/\omega)^{-Z(\omega)-1}$ with $Q$ a polynomial. The singularity analysis of functions with finitely many algebraic singularities on the unit circle \cite[Section~VI.5]{FlajoletSedgewick2009} shows that the contribution of $\omega$ to $T_n$ is of order $n^{\Re Z(\omega)}$, so that the hypothesis makes $\omega=1$ dominant and
\[
T_n\sim\frac{C\,n^{z}}{\Gamma(z+1)},\qquad C=\lim_{x\to1^-}(1-x)^{z+1}T(x).
\]
Since $(1-x)/(1-x^s)\to1/s$, we have $C=s^{-1}\ee^{K}$ with
\[
K=\int_0^1\left(\frac{s\,Z(t)}{t\,(1-t^s)}-\frac{z}{1-t}\right)\dd t=\sum_{i=1}^{r}z_i\int_0^1\left(\frac{s\,t^{i-1}}{1-t^s}-\frac1{1-t}\right)\dd t ,
\]
the integrand being bounded near $t=1$. Splitting $s\,t^{i-1}/(1-t^s)-1/(1-t)$ as $\bigl(s\,t^{i-1}-s\,t^{s-1}\bigr)/(1-t^s)$ plus $s\,t^{s-1}/(1-t^s)-1/(1-t)$, the substitution $u=t^s$ turns the first integral into $\int_0^1(u^{i/s-1}-1)/(1-u)\,\dd u=-\psi(i/s)-\gamma_E$ \cite[5.9.16]{DLMF}, and the second one equals $\lim_{x\to1}\log\bigl((1-x)/(1-x^s)\bigr)=-\log s$. Therefore $\ee^{K}=\prod_ic(s,i)^{-z_i}$ and $T_n\sim n^z/\bigl(s\,\Gamma(z+1)\prod_ic(s,i)^{z_i}\bigr)$, which is the claim.
\end{proof}

As a function of one parameter $z_j$, the others being fixed, the factor $n^{z}$ has slope $1$, and the connection constant $h=1/\bigl(s\,\Gamma(z+1)\prod_ic(s,i)^{z_i}\bigr)$ satisfies
\[
c(s,j)\,(z+1)\,h(z_1,\ldots,z_j+1,\ldots,z_r)=h(z_1,\ldots,z_r).
\]
The new branch of Figure~\ref{fig:triangle} exists here, with an equation of order one whose coefficients lie in $\Q\bigl(c(s,j)\bigr)(z)$. By Gauss's theorem on the digamma function and the recurrence $\psi(x+1)=\psi(x)+1/x$, the number $\psi(j/s)+\gamma_E$ is, for every $j\ge1$, a combination of rational numbers, logarithms, and $\pi\cot(\pi j/s)$ when $j/s\notin\Z$, and equals $\sum_{i=1}^{j/s-1}1/i$ when $j/s\in\Z$, so Euler's constant disappears from $c(s,j)$, for instance $c(2,1)=1/2$ and $c(3,1)=3^{-1/2}\ee^{-\pi/(2\sqrt3)}$. For the family of Proposition~\ref{prop:EG}, the connection constant is $1/\Gamma(z+1)$ multiplied by an exponential in the parameters. When the parameter enters as a shift of $n$, as in the two worlds and in Example~\ref{ex:F}, the connection constants satisfy equations of order two.

The two rules observed for these constants in the note of 2004 \cite[Part~D]{Cloitre2004}, namely $c(ms,mi)=m\,c(s,i)$ and $\prod_{j=0}^{m-1}c(ms,i+js)=c(s,i)^m$ for positive integers $m,s,i$, follow from the closed form of Proposition~\ref{prop:EG} and from the multiplication formula $\sum_{j=0}^{m-1}\psi(x+j/m)=m\bigl(\psi(mx)-\log m\bigr)$, obtained by logarithmic differentiation of \cite[5.5.6]{DLMF}.

\section{Open questions}\label{sec:open}

\begin{enumerate}[label=(\arabic*)]
\item For $c>0$, the recurrence $u_{n+2}=u_{n+1}+c\,u_n/(n+z)$, with $u_1=0$ and $u_2=1$, belongs to the class \eqref{eq:classF} with $\rho=1$ and $\gamma=1+c$. Theorem~\ref{thm:F} gives $n^{-c}u_n(z)\to\ee^{-c}h_c(z)/\Gamma(c+1)$ with $h_c(z)=z\int_0^1t^{z-1}\varphi_c(t)\,\dd t$ and $\varphi_c(t)=(1-t)^{c}\ee^{ct}$, and a functional equation for $h_c$. The case $c=1$ is the world of $\ee$. For $c=1/2$ and $z=0$ one finds $4n/u_n^2\to\pi\ee$, while the value $h_{1/2}(1)=\sqrt{2\pi\ee}\,\operatorname{erf}(1/\sqrt2)-2$ already involves the error function $\operatorname{erf}$. Which characterizations and arithmetic statements extend to this family, and to the analogous deformations of the world of $\pi$?
\item Let $\varphi$ be a D-finite weight and $h$ the moment of Proposition~\ref{prop:D}. When do the functional equation of $h$ and a growth condition determine $h$ up to a constant factor, as in Theorem~\ref{thm:B}, so that the construction of Theorem~\ref{thm:G} becomes canonical? Which integration cycles give the solutions excluded by the growth condition, as the paths of Proposition~\ref{prop:paths} do here?
\item Theorem~\ref{thm:T} bounds the order of the functional equation by the order $r$ of the recurrence, and the equations obtained in the two worlds and in Example~\ref{ex:F} have order two. What is the minimal order in general, how can it be computed, and when is the bound attained? The analytic part of Theorem~\ref{thm:F} should extend to coefficients of higher degree in $n+z$ under assumptions on the dominant singularity, and we leave this extension to another work.
\item Remark~\ref{rem:affine} shows that coefficients which are affine in $n$ and $z$ separately may give a connection function with no functional equation, while the sequence $w_n(z)=n!/(z)_n$ of Euler's Gamma function, which is of this kind, gives one. Which conditions on such recurrences ensure a functional equation, possibly with a step other than $1$?
\item The kernels of Section~\ref{sec:paths} give $f$ and $g$ as Laplace transforms along the positive axis. How do the continuations of these kernels, the jumps between the directions of summation, and the discarded solutions $H$ and $G$ relate to each other?
\end{enumerate}

\section{Sequences of the OEIS}\label{sec:oeis}

The following sequences of the OEIS \cite{OEIS} occur in the paper. The factorials $(m+1)!=a_m$ are A000142, the derangement numbers $D_n$, with $b_m=D_{m+1}$, are A000166, and the numerators of $D_n/n!$ in lowest terms are A053557. The double factorials are A006882. In lowest terms, the numerators and the denominators of $v_{2k}=(2k-1)!!/(2k-2)!!$ are the terms of index $k-1$ of A001803 and A046161. The central binomial coefficients $\binom{2k}{k}$ are A000984, and in lowest terms $p_{2k}=4^k/\binom{2k}{k}$ is the quotient of the terms of index $k$ of A046161 and A001790, which are the denominators and the numerators of $\binom{2k}{k}/4^k$. The sequence A287596 is attached to the case $c=1/2$ of the first open question.

\section*{Use of generative AI}

The research questions, conceptual framework, and core mathematical results (specifically, the mirror recurrences, their complex-parameter extension, and the holonomic triangle formulation) were conceived and established solely by the author. 

Generative AI assistance (Claude Opus~5.5, Anthropic) was employed as a research and drafting tool to explore heuristic edge cases, bibliography, streamline and clarify the presentation of existing proofs, draft scripts for numerical and symbolic cross-verification, and polish the English manuscript. The author independently verified all definitions, intermediate steps, and final proofs, designed and evaluated all computational checks, and assumes full intellectual and scientific responsibility for the manuscript.

\section*{Disclosure statement}

No potential conflict of interest was reported by the author.

\section*{Funding}

No funding was received for this work.

\section*{Data availability statement}

Scripts providing numerical and symbolic checks of selected formulas and examples are available as supplementary material.

\end{document}